\documentclass[11pt]{article}
\usepackage{amsmath,amsthm,amssymb,bbm, color}
\usepackage{geometry}

\def\Ex{{\mathbb E}}
\def\E{{\mathbb E}}
\def\Pr{{\mathbb P}}

\def\er{{\mathbb R}}
\def\R{{\mathbb R}}
\def\te{{\mathbb T}}
\def\zet{{\mathbb Z}}
\def\ind{\mathbbm{1}}
\def\Om{\Omega}
\def\ve{\varepsilon}
\def\de{\; \mathrm{d}}

\newtheorem{thm}{Theorem}
\newtheorem{lem}[thm]{Lemma}
\newtheorem{prop}[thm]{Proposition}
\newtheorem{cor}[thm]{Corollary}
\newcommand*\samethanks[1][\value{footnote}]{\footnotemark[#1]}

\title{Two-sided bounds for $L_p$-norms of combinations of products of independent random variables}
\author{Ewa Damek\thanks{Research supported by the NCN grant DEC-2012/05/B/ST1/00692  and by 
Warsaw Center of Mathematics and Computer Science .}, 
Rafa{\l} Lata{\l}a\thanks{Research supported by the NCN grant DEC-2012/05/B/ST1/00412.},
Piotr Nayar\samethanks\ \  and Tomasz Tkocz}
\date{}

\begin{document}

\maketitle

\begin{abstract}
We show that for every positive $p$, the $L_p$-norm of linear combinations (with scalar or vector coefficients) 
of products of i.i.d.\ random variables, whose
moduli have a nondegenerate distribution with the $p$-norm one, is comparable to the $l_p$-norm of
the coefficients and the constants are explicit. As a result the same holds for linear combinations of Riesz products.
          
We also establish the upper and lower bounds of the $L_p$-moments of partial sums of perpetuities.
\end{abstract}



\section{Introduction and Main Results}

Let $X,X_1,X_2,\ldots$ be i.i.d. nondegenerate nonnegative r.v.'s with finite mean.  Define 
\begin{equation}
\label{defR}
R_0:=1 \quad \mbox{ and }\quad 
R_i:=\prod_{j=1}^{i}X_j \ \mbox{ for }i=1,2,\ldots.
\end{equation}
Then obviously for any vectors $v_0,v_1,\ldots,v_n$ in a normed space $(F,\|\ \|)$, 
$\Ex\|\sum_{i=0}^nv_iR_i\|\leq \sum_{i=0}^n\|v_i\|\Ex R_i$. In \cite{La} it was shown
that the opposite inequality holds, i.e.
\[
\Ex\left\|\sum_{i=0}^nv_iR_i\right\|\geq c_X\sum_{i=0}^n\|v_i\|\Ex R_i,
\]
where $c_X$ is a constant, which depends only on the distribution of $X$.

In this paper we present similar estimates for $L_p$-norms. Our main result is the following.

\begin{thm}
\label{thm_main}
Let $p>0$ and $X,X_1,X_2,\ldots$ be i.i.d. r.v.'s such that $|X|$ is nondegenerate,
$\Ex |X|^p<\infty$ and let $R_i$ be defined by \eqref{defR}. Then there exist constants 
$0<c_{p,X}\leq C_{p,X}<\infty$ which depend only on $p$ and the distribution of $X$ such that
for any  vectors $v_0,v_1,\ldots,v_n$ in a normed space $(F,\|\ \|)$,
\[
c_{p,X}\sum_{i=0}^n\|v_i\|^p\Ex |R_i|^p\leq \Ex\left\|\sum_{i=0}^nv_iR_i\right\|^p
\leq C_{p,X}\sum_{i=0}^n\|v_i\|^p\Ex |R_i|^p.
\]
\end{thm}

\noindent
{\bf Remark.} The assumption that $|X|$ has a nondegenerate distribution is crucial. If $\Pr(X_i=\pm 1)=1/2$ 
then $(R_i)$ are i.i.d. symmetric $\pm 1$ r.v's and  by the Khintchine inequality 
$\Ex|\sum_{i=1}^nR_i|^p$ is of the order $n^{p/2}$, whereas $\sum_{i=1}^n\Ex|R_i|^p=n$.

\medskip

In fact we prove a more general result that does not require the identical distribution assumption.
Namely, suppose that 
\begin{equation}
\label{ass1}
X_1,X_2,\ldots \mbox{ are independent r.v.'s such that }\Ex |X_i|^p<\infty. 
\end{equation}

Further assumptions depend on whether $p\leq 1$. For $p\in (0,1]$ we assume that 
\begin{equation}
\label{ass_small1}
\exists_{\lambda<1}\ \forall_{i}\ \Ex |X_i|^{p/2} \leq \lambda(\Ex |X_i|^p)^{1/2}
\end{equation} 
and 
\begin{equation}
\label{ass_small2}
\exists_{\delta>0, A>1}\ \forall_{i}\ \Ex(|X_i|^p-\Ex |X_i|^p)\ind_{\{ \Ex |X_i|^p \leq |X_i|^p \leq A\Ex |X_i|^p \}} 
\geq \delta\Ex |X_i|^p.
\end{equation} 

\begin{thm}
\label{thm_noniid_smallp}
Let $0<p\leq1$ and $X_1,X_2,\ldots$ satisfy assumptions \eqref{ass1}, \eqref{ass_small1} and \eqref{ass_small2}.
Then for any vectors $v_0,v_1,\ldots,v_n$ in a normed space $(F,\|\ \|)$ we have
\[
c(p,\lambda,\delta,A)\sum_{i=0}^n\|v_i\|^p\Ex |R_i|^p
\leq \Ex\left\|\sum_{i=0}^nv_iR_i\right\|^p
\leq \sum_{i=0}^n\|v_i\|^p\Ex |R_i|^p,
\]
where $c(p,\lambda,\delta,A)$ is a constant which depends only on $p,\lambda,\delta$ and $A$.
\end{thm}

For $p>1$ to obtain the lower bound we assume that
\begin{align}
\notag
\exists_{\mu>0,A<\infty}\ \forall_i\
&\Ex||X_i|-\Ex |X_i|| \geq \mu (\Ex |X_i|^p)^{1/p} 
\\
\label{ass_large1}
&\mbox{ and }\ \Ex||X_i|-\Ex |X_i||\ind_{\{|X_i| > A(\Ex |X_i|^p)^{1/p} \}} 
\leq \frac{1}{4}\mu(\Ex |X_i|^p)^{1/p}
\end{align}
and
\begin{equation}
\label{ass_large2}
\exists_{q>\max\{p-1,1\}}\ \exists_{\lambda<1}\ \forall_{i}\ 
(\Ex |X_i|^{q})^{1/q} \leq \lambda(\Ex |X_i|^p)^{1/p}.
\end{equation}
For the upper bound we need the condition
\begin{equation}
\label{ass_large3}
\forall_{k=1,2,\ldots,\lceil p\rceil -1}\ \exists_{\lambda_k<1}\ \forall_{i}\ 
(\Ex |X_i|^{p-k})^{1/(p-k)} \leq \lambda_k (\Ex |X_i|^{p-k+1})^{1/(p-k+1)}.
\end{equation}

\begin{thm}
\label{thm_noniid_largep}
Let $p>1$ and $X_1,X_2,\ldots$ satisfy assumptions \eqref{ass1}, \eqref{ass_large1}, \eqref{ass_large2} 
and \eqref{ass_large3}.
Then for any vectors $v_0,v_1,\ldots,v_n$ in a normed space $(F,\|\ \|)$ we have
\[
c(p,\mu,A,q,\lambda)\sum_{i=0}^n\|v_i\|^p\Ex |R_i|^p
\leq \Ex\left\|\sum_{i=0}^nv_iR_i\right\|^p
\leq C(p,\lambda_1,\ldots,\lambda_{\lceil p\rceil-1})\sum_{i=0}^n\|v_i\|^p\Ex |R_i|^p,
\]
where $c(p,\mu,A,q,\lambda)$ is a positive constant which depends only on $p,\mu,A,q$ and $\lambda$ 
and $C(p,\lambda_1,\ldots,\lambda_{\lceil p\rceil-1})$ is a constant which depends only
on $p,\lambda_1,\ldots,\lambda_{\lceil p\rceil-1}$.
\end{thm}

\noindent
{\bf Remark.} Proofs presented below show that Theorem \ref{thm_noniid_smallp} holds with
\[
c(p,\lambda,\delta,A)=\frac{\delta^3}{16k},
\mbox{ where $k$ is an integer such that }k\lambda^{2k-2}\leq \frac{\delta^3(1-\lambda)^2}{2^{12} A}. 
\]
In Theorem \ref{thm_noniid_largep} we can take
\[
C(p,\lambda_1,\ldots,\lambda_{\lceil p\rceil-1})
= 2^{\frac{p(p+1)}{2}}\prod_{1\leq j\leq \lceil p\rceil -1}\frac{1}{1-\lambda_j^{p-j}}
\]
and
\[
c(p,\mu,A,q,\lambda)=\frac{\mu^{3p}}{8k\cdot 2^{10p}\cdot 3^p}, \mbox{ where $k$ is an integer such that }
k\lambda^{pk}\leq \frac{(1-\lambda)\mu^{3p}}{8C_0\cdot 2^{10p}\cdot 3^p},
\]
\[
C_0=(1-\lambda)^{1-p}\left(\frac{2A}{3\lambda}\right)^p\left(\frac{2p}{(q+1-p)\ln2}\right)^{\frac{p}{q}}
48^{\frac{2p^2}{\min\{p-1,1\}}}.
\]

\medskip




Another consequence of Theorem \ref{thm_main} is an estimate for $L_p$-norms of linear combinations of the Riesz
products. Let $\te=\er/2\pi\zet$ be the one dimensional torus and $m$ be the normalized Haar measure on $\te$.
The Riesz products are defined on $\te$  by the formula
\[
\bar{R}_i(t)=\prod_{j=1}^i(1+\cos(n_jt)),\quad i=1,2,\ldots,
\]
where $(n_k)_{k\geq 1}$ is a lacunary increasing sequence of positive integers. 

It is well known that if coefficients $n_k$ grow sufficiently fast then 
$\|\sum_{i=0}^na_i\bar{R}_i\|_{L_p(\te)}\sim (\Ex|\sum_{i=0}^na_iR_i|^p)^{1/p}$ for $p\geq 1$, where $R_i$ are products of
independent random variables distributed as $\bar{R}_1$. Together with Theorem \ref{thm_main} this gives
an estimate for $\|\sum_{i=0}^na_i\bar{R}_i\|_{L_p(\te)}$.
Here is the more quantitative result.  

\begin{cor}
Suppose that $(n_k)_{k\geq 1}$ is an increasing sequence of positive integers such that $n_{k+1}/n_k\geq 3$ and
$\sum_{k=1}^{\infty}\frac{n_k}{n_{k+1}}<\infty$.
Then for any coefficients $a_0,a_1,\ldots,a_n\in \er$ and $p\geq 1$,
\[
c_p\sum_{i=0}^n|a_i|^p\int_{\te}|\bar{R}_i(t)|^p \de m(t)\leq 
\int_{\te}\left|\sum_{i=0}^na_i\bar{R}_i(t)\right|^p\de m(t)\leq 
C_p\sum_{i=0}^n|a_i|^p\int_{\te}|\bar{R}_i(t)|^p \de m(t),
\]
where $0<c_p\leq C_p<\infty$ are constants depending only on $p$ and the sequence $(n_k)$. 
\end{cor}

\begin{proof}
Let $X_1,X_2,\ldots$ be independent random variables distributed as $1+\cos(Y)$, where $Y$ is uniformly
distributed on $[0,2\pi]$ and $R_i$ be as in \eqref{defR}. By the result of Y.~Meyer \cite{YM},
$\frac{1}{A}\|\sum_{i=0}^na_i\bar{R}_i\|_{L_p}\leq (\Ex|\sum_{i=0}^na_iR_i|^p)^{1/p}\leq A\|\sum_{i=0}^na_i\bar{R}_i\|_{L_p}$ 
(in particular also $\frac{1}{A}\|\bar{R}_i\|_{L_p}\leq (\Ex R_i^p)^{1/p}\leq A\|\bar{R}_i\|_{L_p}$),
where $A$ depends only on $p$ and the sequence $(n_k)$.
Thus the estimate follows 
by Theorem \ref{thm_main}.
\end{proof}

\bigskip
Theorem \ref{thm_main} has also an immediate application to the stationary $\er^d$--valued solution $S$ of the 
random difference equation
\begin{equation}
\label{diffeq}
S=XS+B,
\end{equation}
where the equality is meant in law and $(X,B)$ is a random variable with values in $[0,\infty)\times \er^d$ independent of $S$ 
such that for some $p>0$, 
\begin{equation}
\tag{GK1} 
\Ex X^p=1,\quad \Ex \| B\| ^p<\infty \quad \mbox{and}\quad \Pr(X=1)<1.  
\end{equation}

Over the last 40 years equation \eqref{diffeq}
and its various modifications have attracted a lot of attention \cite{AB, AM, BBE, BDGHU, DF, G, GG, Gui, HW, K, KP, RS,V}.
It has a wide spectrum of applications including random walks in random environment, branching processes, 
fractals, finance and actuarial mathematics, telecommunications, various physical and biological models. 
In particular, the tail behaviour of $S$ is of interest. 

It is well known that in law
\[
S=\sum _{i=1}^{\infty }R_{i-1}B_i,
\]
where $R_{i-1}=X_1\cdots X_{i-1}, R_0=1$ and $(X_i,B_i)_{i\geq 1}$ is an i.i.d sequence of r.v.'s with the same distribution as $(X,B)$. 
Under the additional assumption that 
\begin{equation}
\tag{GK2} 
\log X \mbox{ conditioned on } \{ X\neq 0\}  \mbox{ is non lattice and } \Ex X^p\log ^+X<\infty,
\end{equation}
$S$ has a heavy tail behaviour, i.e. the limit 
\[
\lim_{t\to \infty}t^p\Pr( \| S\| >t )=c_{\infty}(X,B)
\]
exists and $c_{\infty}(X,B)$ is strictly positive provided that $\Pr(Xv+B=v) <1$ for every 
$v\in \R ^d$. If $\Pr(Xv+B=v) =1$ then $S_n= v-R_{n-1}v\to v=S$. Assumptions (GK1), (GK2) together with 
$\Pr (Xv+B=v) <1$ will be later on referred to as the \emph{Goldie-Kesten conditions}. Let
\[
S_n = \sum _{i=1}^nR_{i-1}B_i.
\] 
It turns out that the sequence $\Ex \| S_n\| ^p$ is closely related to $c_{\infty}(X,B)$. 
Recently, it has been proved in \cite{BDZ} that under the \emph{Goldie-Kesten conditions} plus a little bit 
stronger moment assumption
$\Ex (X^{p+\ve }+\| B\| ^{p+\ve} )<\infty $ for some $\ve >0$, we have
\[
\lim _{n\to \infty }\frac{1}{np\rho}\E \| S_n\| ^p=c_{\infty} (X,B)>0,
\]
where $\rho:= \E X^p\log X$. 

Now suppose that $X,B$ are independent. Then Theorem \ref{thm_main} implies that for every $n$
\begin{equation}
\label{boundsp}
c_{p,X}\Ex \| B\| ^p\leq \frac{1}{n}\Ex \| S_n\| ^p\leq C_{p,X} \Ex \| B\| ^p,
\end{equation}
which gives uniform bounds on the Goldie constant $c_{\infty} (X,B)$ depending only on the law of $X$ and 
$\Ex \| B\| ^p$ and independent of the dimension. Moreover, in some particular cases when constants 
$\lambda , \delta, \mu , q, \lambda_k$ in  \eqref{ass_small1}--\eqref{ass_large3} can be estimated more carefully, 
\eqref{boundsp} may give  some information about the size of the Goldie constant which is of some value, especially in the situation when none of the
existing formulae for it is satisfactory enough (see \cite{CV, ESZ, BDZ, BJMW}).

We can go even further. With a slight modification of the proof we can get rid of independence of $X,B$ and obtain 
the following theorem.

\begin{thm}
\label{mthmper}
Suppose that $F$ is a separable Banach space. Let $p>0$ and let an i.i.d. sequence 
$(X,B), (X_1,B_1),...$ with values in $[0,\infty )\times F$ 
be such that $X$ is nondegenerate and $\Ex \| B\| ^p, \Ex X^p <\infty$. Assume additionally 
that 
\begin{equation}
\label{eq:nondeg}
\Pr (Xv+B=v) <1 \mbox{ for every }v\in F. 
\end{equation}
Then there are constants  $c_p(X,B)>0$ which depend on $p$ and the distribution of $(X,B)$
and $C_p(X)<\infty$ which depend on $p$ and the distribution of $X$
such that for every $n$,
\begin{equation}
\label{bounds}
c_p(X,B)\E \| B\| ^p\sum _{i=1}^n \Ex R_{i-1}^p\leq \Ex \left\| \sum _{i=1}^n R_{i-1} B_i\right\| ^p
\leq C_p(X) \Ex \| B\| ^p \sum _{i=1}^n \Ex R_{i-1}^p.
\end{equation}
\end{thm}

Theorem \ref{mthmper} specified to our situation with $\Ex X^p=1$ gives
\[
c_p(X,B)\Ex \| B\| ^p\leq \frac{1}{n}\Ex \| S_n\| ^p\leq C_p(X) \Ex \| B\| ^p.
\]
This leads to an estimate for the Goldie constant but now with  $c_p(X,B), C_p(X)$ depending on the law of $(X,B)$. 
Again, in particular cases, a  careful examination of the constants involved in the proof may give a more satisfactory 
answer. Also, in view of Theorem \ref{mthmper}, it would be worth relaxing the assumptions of \cite{BDZ}.

\medskip

The paper is organized as follows. In Section \ref{sec:lowerlarge} and \ref{sec:lowersmall} we derive lower bounds in 
Theorems \ref{thm_noniid_smallp} and \ref{thm_noniid_largep}. Then in Section \ref{sec:upper} we establish upper bounds in
both theorems. We conclude in Section \ref{sec:perp} with a discussion of the proof of Theorem \ref{mthmper}.

\section{Lower bound for $p>1$}
\label{sec:lowerlarge}

In this section we will show the lower bound in Theorem \ref{thm_noniid_largep}. Since it is only a matter
of normalization we will assume that
\begin{equation}
\label{ass1a}
X_1,X_2,\ldots \mbox{ are independent r.v.'s such that }\Ex |X_i|^p=1. 
\end{equation}
In particular this implies that $\Ex |R_i|^p=1$ for all $i$.

We also set for $k=1,2,\ldots$
\[
R_{k,k-1}\equiv 1 \quad \mbox{and}\quad R_{k,i}:=\prod_{j=k}^i X_i \mbox{ for }i\geq k.
\]
Observe that $R_i=R_kR_{k+1,i}$ for $i\geq k\geq 0$.

We begin with several lemmas.

\begin{lem}
\label{lem:1}
Suppose that a r.v.\ $X$ satisfies $\Ex||X|-\Ex |X|| \geq \mu$ and 
$\Ex||X|-\Ex |X||\ind_{\{|X| > A\}} \leq \frac{1}{4} \mu$. Then for all $p \geq 1$ and $u,v \in (F,\|\ \|)$ we have
\[
\Ex \|uX+v\|^p\geq \Ex\|uX+v\|^p \ind_{\{ |X| \leq A\}} 
\geq 
\frac{\mu^p}{8^p} \min\left\{ 1, \frac{1}{(\Ex |X|)^p}  \right\} \max \{ \|u\|^p,\|v\|^p \}.		
\] 
\end{lem}

\begin{proof}
Let $Y$ has the same distribution as $X$ conditioned on the set $\{|X| \leq A \}$. Let us define $t:=\Ex Y$. 
Then $|t|\leq \Ex|Y| \leq \Ex |X|$. Clearly, $\Ex(|X|-\Ex |X|)_+ = \Ex(|X|-\Ex |X|)_- \geq \frac{1}{2} \mu$.
Therefore,
\begin{align*}
\Ex|X-t|\ind_{\{ |X| \leq A \}} 
& \geq \Ex ||X|-|t|| \ind_{\{ |X| \leq A \}}
\geq \Ex(|X|-|t|)_+ \ind_{\{ |X| \leq A \}}  
\\
&\geq \Ex(|X|-\Ex |X|)_+ \ind_{\{ |X| \leq A \}} 
\\
&= \Ex(|X|-\Ex |X|)_+ - \Ex(|X|-\Ex |X|)_+ \ind_{\{ |X| > A \}} 
\\ 
&\geq  \frac{1}{2} \mu -  \Ex||X|-\Ex |X|| \ind_{\{ |X| > A \}} \geq \frac{1}{2}\mu - \frac{1}{4}\mu = \frac{1}{4}\mu. 
\end{align*}

We obtain
\[
|t|\Ex \|uY+v\| =\Ex \|v(t-Y) + (tu+v)Y\| \geq \|v\|\Ex|Y-t| - \|tu+v\|\Ex|Y|. 
\]
Since $\Ex \|uY+v\| \geq \|u\Ex Y+v\|= \|tu+v\|$ and $|t|\leq \Ex|Y|\leq \Ex|X|$ we have
\begin{align*}
\Ex \|uY+v\| 
&\geq \frac{1}{|t|+\Ex|Y|} \|v\| \Ex|Y-t|
\geq \frac{\|v\|}{2\Ex|X| \Pr(|X| \leq A)} \Ex |X-t| \ind_{\{ |X| \leq A\}}
\\ 
&\geq \frac{\mu}{8\Ex|X|} \frac{\|v\|}{\Pr(|X| \leq A)}.  
\end{align*}
We arrive at
\begin{align*}
\Ex \|uX+v\|^p \ind_{\{ |X| \leq A\}} 
& \geq  \left( \Ex \|uX+v\| \ind_{\{ |X| \leq A\}} \right)^p  = \left( \Ex \|uY+v\| \Pr(|X| \leq A)  \right)^p 
\\ 
&\geq \frac{\mu^p}{8^p (\Ex |X|)^p}\|v\|^p. 
\end{align*}

We also have
\[
\Ex \|uY+v\| = \Ex \| u(Y-t) + tu+v \| \geq \|u\| \Ex|Y-t| - \|tu+v\|.
\]
Therefore 
\[
\Ex \|uY+v\| \geq \frac{\|u\|}{2} \Ex|Y-t| \geq \frac{\mu}{8} \frac{\|u\|}{\Pr(X \leq A)}
\]
and as before we get that $\Ex \|uX+v\|^p \ind_{\{ X \leq A\}} \geq \frac{\mu^p}{8^p}  \|u\|^p$.		
\end{proof}

\begin{lem}
\label{lem:3}
Assume that \eqref{ass1a} and \eqref{ass_large1} hold. Then for any $v_0, v_1,\ldots, v_n \in (F,\|\ \|)$ we have
\[
\Ex \left\|\sum_{i=0}^n v_i R_i\right\|^p \geq \frac{\mu^{2p}}{64^p} \max_{1 \leq i \leq n} \|v_i\|^p 
\geq \frac{\mu^{2p}}{64^p} \cdot \frac{1}{n} \sum_{i=1}^n \|v_i\|^p.
\]
\end{lem}

\begin{proof}
For $1 \leq j \leq n$ we have $\sum_{i=0}^n v_i R_i = Y + X_j(v_jR_{j-1} + X_{j+1} Z)$, where $Y$ and $Z$ are 
independent of $X_j$ and $X_{j+1}$. Observe that $\Ex |X_j| \leq 1$ and $\Ex |X_{j+1}| \leq 1$. 
Thus, using Lemma \ref{lem:1} twice, we obtain
\[
\Ex \left\|\sum_{i=0}^n v_i R_i\right\|^p \geq \frac{\mu^p}{8^p} \Ex \|v_jR_{j-1} + X_{j+1} Z\|^p 
\geq \frac{\mu^{2p}}{64^p} \|v_j\|^p \Ex |R_{j-1}|^p = \frac{\mu^{2p}}{64^p} \|v_j\|^p.  
\]
\end{proof}

\begin{lem}
\label{lem:2}
Assume that \eqref{ass1a} holds and there exist $q>1$ and $0<\lambda<1$ such that for all $i$, 
$(\Ex |X_i|^q)^{1/q}\leq \lambda$. Then for any $v_0,v_1,\ldots,v_n\in (F,\|\ \|)$ and $t>0$,
\[
\Pr\left(\left\|\sum_{i=0}^n v_i R_i\right\|^p \geq t\sum_{i=0}^n \lambda^i \|v_i\|^p  \right) 
\leq (1-\lambda)^{\frac{(1-p)q}{p}} t^{-\frac{q}{p}}.
\]
\end{lem}

\begin{proof}
Using Minkowski's and H\"older's inequalities we obtain
\begin{align*}
\left(\Ex \left\|\sum_{i=0}^n  v_i R_i \right\|^{q}\right)^{\frac{1}{q}} 
& \leq \sum_{i=0}^n \left( \Ex \|v_i R_i\|^{q} \right)^{\frac{1}{q}} 
\leq \sum_{i=0}^n \|v_i\| \lambda^i 
= \sum_{i=0}^n \|v_i\| \lambda^{\frac{i}{p}} \lambda^{\frac{p-1}{p}i} 
\\ 
& \leq \left( \sum_{i=0}^n \|v_i\|^p \lambda^i \right)^{\frac{1}{p}} 
\left(\sum_{i=0}^n \lambda^i\right)^{\frac{p-1}{p}} . 
\end{align*}
Thus,
\[
\Ex \left\|\sum_{i=0}^n  v_i R_i \right\|^{q} \leq 
\left(\sum_{i=0}^n \|v_i\|^p \lambda^i  \right)^{\frac{q}{p}}(1-\lambda)^{-\frac{(p-1)q}{p}}.
\]
By Chebyshev's inequality we get
\[
\Pr\left(\left\|\sum_{i=0}^n v_i R_i\right\|^{q} \geq 
t^{\frac{q}{p}} \left(\sum_{i=0}^n \lambda^i \|v_i\|^p \right)^{\frac{q}{p}}\right)
\leq (1-\lambda)^{\frac{(1-p)q}{p}} t^{-\frac{q}{p}}. 
\]
\end{proof}

\begin{lem}
\label{lem:4}
Let $Y,Z$ be random vectors with values in a normed space $F$ and let $p \geq 1$. 
Suppose that $\Ex\|Y\|^{p-1} \|Z\|   \leq \gamma \Ex\|Z\|^p$. Then 
\[
\Ex\|Y+Z\|^p \geq \Ex \|Y\|^p + \left( \frac{1}{3^p}-2p \gamma \right) \Ex\|Z\|^p.
\]    
\end{lem}

\begin{proof}
For any real numbers $a,b$ we have $|a+b|^p \geq |a|^p -p|a|^{p-1}|b|$. If, additionally, $|a| \leq \frac{1}{3}|b|$,
then $|a+b|^p \geq |a|^p + \frac{1}{3^p} |b|^p$. Taking $a= \|Y\|$, $b=-\|Z\|$ and using the inequality 
$\|Y+Z\| \geq |\|Y\|-\|Z\||$ we obtain
\begin{align*}
\Ex\|Y+Z\|^p 
& = \Ex\|Y+Z\|^p \ind_{\{\|Y\| \leq \frac{1}{3}\|Z\|\}} + \Ex\|Y+Z\|^p \ind_{\{\|Y\| > \frac{1}{3} \|Z\|\}} 
\\
& \geq  \Ex\|Y\|^p \ind_{\{\|Y\| \leq \frac{1}{3} \|Z\|\}}+ \frac{1}{3^p}\Ex\|Z\|^p\ind_{\{\|Y\|\leq\frac{1}{3}\|Z\|\}}
\\ 
&\phantom{aa}+ \Ex\|Y\|^p\ind_{\{\|Y\|>\frac{1}{3} \|Z\|\}} - p\Ex\|Y\|^{p-1}\|Z\|\ind_{\{\|Y\|>\frac{1}{3}\|Z\|\}} 
\\
& = \Ex\|Y\|^p + \frac{1}{3^p}\Ex\|Z\|^p(1-\ind_{\{\|Y\| > \frac{1}{3} \|Z\|\}})
- p \Ex\|Y\|^{p-1} \|Z\| \ind_{\{\|Y\| > \frac{1}{3} \|Z\|\}}.
\end{align*}
Note that
\[
\Ex\left(\frac{1}{3^p}\|Z\|^p + p \|Y\|^{p-1} \|Z\|\right) \ind_{\{\|Y\| > \frac{1}{3} \|Z\|\}} 
\leq \left(\frac{1}{3} + p \right) \Ex\|Y\|^{p-1} \|Z\| \leq 2p  \gamma \Ex\|Z\|^p . 
\]
Therefore,
\[
\Ex\|Y+Z\|^p \geq \Ex\|Y\|^p + \frac{1}{3^p} \Ex\|Z\|^p - 2p \gamma \Ex \|Z\|^p.
\] 
\end{proof}

We are now able to state the key proposition which will easily yield the lower bound in 
Theorem \ref{thm_noniid_largep}.

\begin{prop}
\label{prop:1}
Let $p>1$ and suppose that r.v.'s $X_1,X_2,\ldots$ satisfy assumptions \eqref{ass1a}, \eqref{ass_large1} and \eqref{ass_large2}. 
Then there exist constants $\ve_0, \ve_1, C_0>0$ depending only on $p, \mu, A, q$ and $\lambda$ such that
for any vectors  $v_0,v_1,\ldots,v_n$ in a normed space $(F,\|\ \|)$ and $k\geq 1$ we have
\begin{equation}
\label{lower_largep}
\Ex \left\|\sum_{i=0}^n v_i R_i\right\|^p 
\geq \ve_0 \|v_0\|^p + \sum_{i=1}^n \left(\frac{\ve_1}{k}-c_i\right) \|v_i\|^p,
\end{equation}
where
\[
c_i = 0 \ \mbox{ for }  1 \leq i \leq k-1, \quad c_i = \Phi \sum_{j=k}^i \lambda^j  \mbox{ for} \ i \geq k \quad  
\mbox{and} \quad \Phi = C_0 \lambda^{(p-1) k}.
\]
\end{prop}

\begin{proof}
Define
\[
\ve_0:=\min\left\{\frac{1}{4\cdot 3^p},\frac{\mu^p}{8\cdot 24^p}\right\}, \quad 
\ve_1:=\min\left\{\frac{\mu^p}{8^p},\frac{\mu^{2p}}{2^{p-1}64^p}\right\}\ve_0,
\]
where the value of $C_0$ will be chosen later. 
In the proof by $\ve_2,C_2,C_3$ we denote finite nonnegative constants that depend only on 
parameters $p,\mu.A,q$ and $\lambda$.

We fix $k\geq 1$ and prove \eqref{lower_largep} by induction on $n$. From Lemma \ref{lem:1} and Lemma \ref{lem:3} 
we obtain 
\[
\Ex \left\|\sum_{i=0}^n v_i R_i \right\|^p \geq 2\ve_0 \|v_0\|^p, \qquad 
\Ex \left\|\sum_{i=0}^n v_i R_i \right\|^p \geq  \frac{2\ve_1}{n} \sum_{i=1}^n \|v_i\|^p.
\]
Therefore for $n \leq k$ we have
\[
\Ex \left\|\sum_{i=0}^n v_i R_i\right\|^p \geq  \ve_0 \|v_0\|^p + \frac{\ve_1}{k} \sum_{i=1}^n \|v_i\|^p.
\]

Suppose that the induction assertion holds for $n \geq k$. We show it for $n+1$. To this end we consider two cases.

\medskip

\noindent 
\textbf{Case 1.} $\ve_0 \|v_0\|^p \leq \Phi \sum_{i=k}^{n+1}\lambda^i \|v_i\|^p $.

Applying the induction assumption conditionally on $X_1$ we obtain
\begin{align*}
\Ex \left\|\sum_{i=0}^{n+1} v_i R_i\right\|^p 
& \geq \ve_0 \Ex\|v_0+v_1X_1\|^p + \sum_{i=2}^{n+1} \left(\frac{\ve_1}{k}-c_{i-1}\right) \Ex \|X_1 v_i\|^p 
\\
& \geq \frac{\ve_1}{k} \|v_1\|^p + \sum_{i=2}^{n+1} \left(\frac{\ve_1}{k}-c_{i-1}\right)\|v_i\|^p 
\\ 
& \geq  \ve_0 \|v_0\|^p - \Phi \sum_{i=k}^{n+1}\lambda^i \|v_i\|^p + \frac{\ve_1}{k}\|v_1\|^p 
+ \sum_{i=2}^{n+1} \left(\frac{\ve_1}{k}-c_{i-1}\right)\|v_i\|^p 
\\
& =  \ve_0 \|v_0\|^p + \sum_{i=1}^{n+1}\left(\frac{\ve_1}{k}-c_{i}\right)\|v_i\|^p,
\end{align*}
where the second inequality follows from Lemma \ref{lem:1}. 

\medskip
\noindent 
\textbf{Case 2.} $\ve_0 \|v_0\|^p > \Phi \sum_{i=k}^{n+1}\lambda^i \|v_i\|^p $.

Define the event $A_k \in \sigma(X_1,\dots,X_k)$ by
\[
A_k := \{ |X_1| \leq A, |R_{2,k}| \leq 2^{\frac{1}{q}} \lambda^{k-1} \}.
\]
By the induction assumption used conditionally on $X_1,\ldots,X_k$ we have
\begin{equation}
\label{row1}
\Ex \left\|\sum_{i=0}^{n+1} v_i R_i\right\|^p \ind_{\Omega \setminus A_k} 
\geq  \ve_0 \Ex \left\|\sum_{i=0}^{k} v_i R_i\right\|^p \ind_{\Omega \setminus A_k} 
+ \sum_{i=k+1}^{n+1} \left(\frac{\ve_1}{k}-c_{i-k}\right) \Ex \|v_i R_k\|^p \ind_{\Omega \setminus A_k}.
\end{equation}

We have by Chebyshev's inequality and \eqref{ass_large2},
\begin{equation}
\label{Rk} 
\Pr\left(|R_{2,k}| \geq 2^{\frac{1}{q}} \lambda^{k-1}\right) 
\leq  \frac{\Ex |R_{2,k}|^{q}}{2 \lambda^{(k-1)q}} \leq \frac{1}{2}. 
\end{equation}
Together with \eqref{ass_large1} it implies $\Pr(A_k)>0$.
Let  $(Y,Y',Z)$ have the same distribution as the random vector 
$(\sum_{i=k}^{n+1} v_i R_i,\sum_{i=k}^{n+1} v_i R_{k+1,i},\sum_{i=0}^{k-1} v_i R_i)$
conditioned on the event $A_k$. 
Note that
\[
\Ex \left\|\sum_{i=0}^{n+1} v_i R_i\right\|^p \ind_{ A_k}  = \Pr(A_k) \Ex \|Y+Z\|^p.
\]
 Applying Lemma \ref{lem:1} conditionally we obtain 
\begin{align}
\notag
\Ex\|Z\|^p  
&= \frac{1}{\Pr(A_k)} \Ex\left\|\sum_{i=0}^{k-1}v_i R_i\right\|^p
\ind_{\{|X_1| \leq A\}}\ind_{\{ |R_{2,k}|\leq 2^{\frac{1}{q}}\lambda^{k-1}\}} 
\\
\label{zv0}
&\geq \frac{\mu^p}{8^p} \|v_0\|^p \frac{\Pr( |R_{2,k}| \leq  2^{\frac{1}{q}} \lambda^{k-1})}{\Pr(A_k)} 
= \frac{\mu^p}{8^p} \|v_0\|^p \frac{1}{\Pr(|X_1| \leq A)} \geq \frac{\mu^p}{8^p} \|v_0\|^p .
\end{align}
Note that $Y'$ has the same distribution as $\sum_{i=k}^{n+1} v_i R_{k+1,i}$ and is independent of $Z$. 
We have for $t>0$,
\begin{align}
\notag
\Pr(\|Y\|^p \geq t \Ex \|Z\|^p )  
&\leq  \Pr\left(A^p \lambda^{p(k-1)} 2^{\frac{p}{q}} \|Y'\|^p \geq t \frac{\mu^p}{8^p} \|v_0\|^p \right) 
\\
\notag
&\leq \Pr\left(A^p\lambda^{p(k-1)} 2^{\frac{p}{q}} \|Y'\|^p \geq t \frac{\mu^p}{8^p} \frac{\Phi}{\ve_0} 
\sum_{i=k}^{n+1}\lambda^i \|v_i\|^p\right) 
\\
\label{tail}
&=\Pr\left(\|Y'\|^p \geq t C_0 \ve_2 \sum_{i=k}^{n+1}\lambda^{i-k} \|v_i\|^p\right)  
\leq C_1 (tC_0)^{-\frac{q}{p}},
\end{align}
where the last inequality follows by Lemma \ref{lem:2}
(recall that $\ve_2$ and $C_1$ denote constants depending on $p,\mu,A,q$ and $\lambda$).

In order to use Lemma \ref{lem:4} we would like to estimate $\Ex\|Y\|^{p-1}\|Z\|$. To this end take $\delta > 0$
and observe first that
\begin{align}
\notag
\Ex \|Y\|^{p-1} \|Z\|
& \leq \Ex \|Y\|^{p-1} \|Z\| \ind_{\{\|Y\|^p \leq \delta \Ex\|Z\|^p \}} 
+ \Ex \|Y\|^{p-1} \|Z\| \ind_{\{ \|Z\|^p \leq \delta \Ex \|Z\|^p \}}
\\
\label{eqq0}
&\phantom{aa}+\Ex\|Y\|^{p-1} \|Z\|  \ind_{\{ \|Y\|^p > \delta \Ex \|Z\|^p \}} \ind_{\{ \|Z\|^p > \delta \Ex \|Z\|^p \}}.
\end{align}
Clearly,
\begin{equation}
\label{eqq1}
\Ex  \|Y\|^{p-1} \|Z\| \ind_{\{ \|Y\|^p \leq \delta \Ex \|Z\|^p \}} 
\leq \delta^{\frac{p-1}{p}} (\Ex\|Z\|^p)^{\frac{p-1}{p}} \Ex \|Z\| 
\leq  \delta^{\frac{p-1}{p}} \Ex \|Z\|^p.
\end{equation}

To estimate the next term in \eqref{eqq0} note that
\[
\Ex\|Y\|^{p-1} \|Z\| \ind_{\{ \|Z\|^p  \leq \delta \Ex \|Z\|^p \}}  
\leq \delta^{1/p}  (\Ex \|Z\|^p)^{1/p}  \Ex \|Y\|^{p-1}.
\]
Using estimate \eqref{tail} we obtain
\begin{align*}
\Ex \|Y\|^{p-1} 
& =  (\Ex \|Z\|^p)^{\frac{p-1}{p}}\int_{0}^{\infty}\Pr\left(\|Y\|^{p} \geq s^{\frac{p}{p-1}} \Ex\|Z\|^p\right)\de s 
\\ 
&\leq (\Ex \|Z\|^p)^{\frac{p-1}{p}}\int_{0}^{\infty}\min\{1, C_1C_0^{-\frac{q}{p}} s^{-\frac{q}{p-1}}\}\de s  
\leq   (\Ex \|Z\|^p)^{\frac{p-1}{p}}\left(1+ C_2 C_0^{-\frac{q}{p}} \right),
\end{align*}
where the last inequality follows since $q>p-1$.
Thus, 
\begin{equation}
\label{eqq2}
\Ex \|Y\|^{p-1} \|Z\| \ind_{\{ \|Z\|^p  \leq \delta \Ex \|Z\|^p \}}  
\leq \delta^{1/p} \left(  1+ C_2 C_0^{-\frac{q}{p}}\right)  \Ex \|Z\|^p.
\end{equation}

We are left with estimating the last term in \eqref{eqq0}. We have
\begin{align*}
\Ex \|Y\|^{p-1}& \|Z\| \ind_{\{ \|Y\|^p > \delta \Ex \|Z\|^p \}} \ind_{\{ \|Z\|^p > \delta \Ex \|Z\|^p \}} 
\\ 
& = \sum_{m=0}^\infty \Ex \|Y\|^{p-1} \|Z\| \ind_{\{ 2^m \delta \Ex \|Z\|^p < \|Y\|^p \leq 2^{m+1}\delta \Ex\|Z\|^p\}}
\ind_{\{ \|Z\|^p > \delta \Ex \|Z\|^p \}} 
\\
& \leq \sum_{m=0}^\infty 2^{(m+1)\frac{p-1}{p}} \delta^{\frac{p-1}{p}}
\Ex(\Ex\|Z\|^p)^{\frac{p-1}{p}} \|Z\|\ind_{\{2^m\delta\Ex\|Z\|^p<\|Y\|^p\}}\ind_{\{\|Z\|^p>\delta \Ex \|Z\|^p \}} 
\\
& \leq \delta^{\frac{p-1}{p}} \sum_{m=0}^\infty  2^{(m+1)\frac{p-1}{p}} 
\Ex \left( \frac{\|Z\|^p}{\delta} \right)^{\frac{p-1}{p}} \|Z\| \ind_{\{ 2^m \delta \Ex \|Z\|^p < \|Y\|^p \}} 
\\
& =  \sum_{m=0}^\infty  2^{(m+1)\frac{p-1}{p}}  \Ex \|Z\|^p \ind_{\{ 2^m \delta \Ex \|Z\|^p < \|Y\|^p \}}.
\end{align*} 
Recall that $Z$ and $Y'$ are independent. Therefore as in \eqref{tail} we get
\begin{align*}
\Ex\|Z\|^p \ind_{\{ 2^m \delta \Ex \|Z\|^p < \|Y\|^p  \}} 
& \leq \Ex \|Z\|^p 
\ind_{\{ \|Y'\|^p \geq 2^m \delta C_0 \ve_2 \sum_{i=k}^{n+1}\lambda^{i-k} \|v_i\|^p \}} 
\\
& = \Ex \|Z\|^p
\Pr\left(\|Y'\|^p \geq 2^m \delta C_0 \ve_2 \sum_{i=k}^{n+1}\lambda^{i-k} \|v_i\|^p \right) 
\\
& \leq  \Ex \|Z\|^p C_1 (2^m \delta C_0)^{-\frac{q}{p}}.
\end{align*}
We arrive at
\begin{align}
\notag
\Ex \|Y\|^{p-1} \|Z\| \ind_{\{ \|Y\|^p > \delta \Ex \|Z\|^p \}} \ind_{\{ \|Z\|^p > \delta \Ex \|Z\|^p \}}  
&  \leq \Ex \|Z\|^p  C_1     (\delta C_0)^{-\frac{q}{p}}\sum_{m=0}^\infty  2^{(m+1)\frac{p-1}{p}}  2^{-\frac{mq}{p}} 
\\
\label{eqq3}
&  \leq  \Ex \|Z\|^p C_3 (\delta C_0)^{-\frac{q}{p}},
\end{align}
where we have used the fact that $q>p-1$.

Estimates \eqref{eqq0}--\eqref{eqq3} imply
\[
\Ex\|Y\|^{p-1} \|Z\|  \leq \Ex \|Z\|^p  
\left( \delta^{\frac{p-1}{p}} + \delta^{1/p}( 1+ C_2C_0^{-\frac{q}{p}}) + C_3(\delta C_0)^{-\frac{q}{p}}\right).
\]
Now we choose $\delta=\delta(p)$ sufficiently small and then $C_0=C_0(p,A,\mu.q,\lambda)$ sufficiently large 
to obtain
\begin{equation}
\label{eqq4}
\Ex\|Y\|^{p-1} \|Z\|   \leq  \frac{1}{4p 3^p}\Ex \|Z\|^p.
\end{equation}

>From Lemma \ref{lem:4} we deduce
\[
\Ex\|Y+Z\|^p \geq \Ex \|Y\|^p +  \frac{1}{2 \cdot 3^p} \Ex\|Z\|^p.
\]
Hence
\begin{equation}
\label{fineq0}
\Ex \left\| \sum_{i=0}^{n+1} v_i R_i \right\|^p \ind_{A_k} 
\geq \frac{1}{2 \cdot 3^p} \Ex \left\|\sum_{i=0}^{k-1} v_i R_i\right\|^p \ind_{A_k} 
+ \Ex \left\|\sum_{i=k}^{n+1} v_i R_i\right\|^p \ind_{A_k}.
\end{equation}
Lemma \ref{lem:1} and \eqref{Rk} yield
\[
\Ex \left\|\sum_{i=0}^{k-1} v_i R_i\right\|^p \ind_{A_k}  
\geq  \frac{\mu^p}{8^p} \|v_0\|^p \Pr(|R_{2,k}| \leq 2^{\frac{1}{q}} \lambda^{k-1}) 
\geq \frac{1}{2} \cdot  \frac{\mu^p}{8^p} \|v_0\|^p.
\]
It follows that
\begin{equation}
\label{fineq1}
\frac{1}{2 \cdot 3^p} \Ex \left\|\sum_{i=0}^{k-1} v_i R_i\right\|^p \ind_{A_k}  
\geq  \ve_0  \|v_0\|^p + \ve_0 \Ex \left\|\sum_{i=0}^{k-1} v_i R_i\right\|^p \ind_{A_k}.
\end{equation}
By the induction assumption we obtain
\begin{equation}
\label{fineq2}
\Ex \left\|\sum_{i=k}^{n+1} v_i R_i\right\|^p \ind_{A_k} \geq \ve_0 \Ex \|v_k R_k\|^p \ind_{A_k} 
+ \sum_{i=k+1}^{n+1} \left(\frac{\ve_1}{k} - c_{i-k}\right) \Ex \|v_i R_k\|^p \ind_{A_k}.
\end{equation}

Combining \eqref{fineq0}, \eqref{fineq1} and \eqref{fineq2} we get
\begin{align*}
\Ex\left\| \sum_{i=0}^{n+1} v_i R_i \right\|^p \ind_{A_k}  
&\geq \ve_0\|v_0\|^p  + \ve_0 \Ex \left\|\sum_{i=0}^{k-1} v_i R_i\right\|^p \ind_{A_k} 
+ \ve_0 \Ex \|v_k R_k\|^p \ind_{A_k} 
\\ 
&  \phantom{aa}  + \sum_{i=k+1}^{n+1} \left(\frac{\ve_1}{k} - c_{i-k}\right) \Ex \|v_i R_k\|^p \ind_{A_k} 
\\
&   \geq \ve_0 \|v_0\|^p +  \frac{\ve_0}{2^{p-1}} \Ex \left\|\sum_{i=0}^{k} v_i R_i\right\|^p \ind_{A_k}  
+\sum_{i=k+1}^{n+1} \left(\frac{\ve_1}{k} - c_{i-k}\right) \Ex \|v_i R_k\|^p \ind_{A_k}.
\end{align*}
This inequality together with \eqref{row1} and Lemma \ref{lem:3} yields
\begin{align*}
\Ex\left\| \sum_{i=0}^{n+1} v_i R_i\right\|^p 
& \geq \ve_0 \|v_0\|^p + \frac{\ve_0}{2^{p-1}} \Ex \left\|\sum_{i=0}^{k} v_i R_i\right\|^p  
+\sum_{i=k+1}^{n+1} \left(\frac{\ve_1}{k} - c_{i-k}\right) \Ex \|v_i R_k\|^p    
\\
& \geq  \ve_0 \|v_0\|^p+ \frac{\ve_1}{k} \sum_{i=1}^{k} \|v_i\|^p 
+\sum_{i=k+1}^{n+1}  \left(\frac{\ve_1}{k} - c_{i-k}\right) \|v_i\|^p 
\\
& \geq \ve_0 \|v_0\|^p  +\sum_{i=1}^{n+1} \left(\frac{\ve_1}{k} - c_{i}\right) \|v_i\|^p.
\end{align*}
\end{proof}

We are ready to prove the lower $L_p$-estimate for $p>1$.

\begin{proof}[Proof of the lower bound in Theorem \ref{thm_noniid_largep}]
For sufficiently large $k$ we have for all $i$, 
\[
c_i \leq \frac{\Phi\lambda^k}{1-\lambda} = \frac{C_0 \lambda^{p k}}{1-\lambda} \leq \frac{\ve_1}{2k}.
\]
Thus, Proposition \ref{prop:1} yields
\[
\Ex \left\| \sum_{i=0}^n v_i R_i \right\|^p \geq \ve_0 \|v_0\|^p 
+ \frac{\ve_1}{2k} \sum_{i=1}^n  \|v_i\|^p \geq \ve \sum_{i=0}^n  \|v_i\|^p,
\] 
where $\ve:=\min\{\ve_0,\frac{\ve_1}{2k}\}$.
\end{proof} 

\noindent
{\bf Remark.} Observe that $\mu\leq \Ex||X_i|-\Ex |X_i||\leq 2\Ex |X_i|\leq 2(\Ex |X_i|^p)^{1/p}=2$. This shows that
\[
\ve_0=\frac{\mu^p}{8\cdot 24^p},\quad \ve_1=\frac{\mu^{2p}}{2^{p-1}64^p}\cdot \ve_0
\quad \mbox{and} \quad \min\left\{\ve_0,\frac{\ve_1}{2k}\right\}=\frac{\mu^{3p}}{8k\cdot 2^{10p}\cdot 3^p}.
\]
Other constants used in the proof of Proposition \ref{prop:1} may be estimated as follows 
\[
\ve_2=\frac{\mu^p\lambda^p}{8^p A^p\ve_0}2^{-\frac{p}{q}}\geq \left(\frac{3\lambda}{2A}\right)^p,
\quad C_1=(1-\lambda)^{\frac{(1-p)q}{p}}\ve_2^{-\frac{q}{p}}\leq 
(1-\lambda)^{\frac{(1-p)q}{p}}\left(\frac{2A}{3\lambda}\right)^q,
\]
\[
C_2\leq \frac{p-1}{q+1-p}C_1\quad \mbox{ and }\quad C_3=\frac{2^{\frac{q}{p}}}{2^{\frac{q+1-p}{p}}-1}C_1\leq
\frac{2p}{(q+1-p)\ln 2}C_1.
\]
Hence we can for example take 
\[
\delta:=48^{-\frac{p^2}{\min\{p-1,1\}}}\quad \mbox{and} \quad
C_0:=(1-\lambda)^{1-p}\left(\frac{2A}{3\lambda}\right)^p\left(\frac{2p}{(q+1-p)\ln2}\right)^{\frac{p}{q}}
48^{\frac{2p^2}{\min\{p-1,1\}}},
\]
then each term $\delta^{(p-1)/p}$, $\delta^{1/p}$, $\delta^{1/p}C_2C_0^{-q/p}$ and
$C_3(\delta C_0)^{-q/p}$ is not greater than $48^{-p}\leq (16p3^p)^{-1}$ and \eqref{eqq4} holds.

\section{Lower bound for $p\leq 1$}
\label{sec:lowersmall}

In this section we prove the lower bound in Theorem \ref{thm_noniid_smallp}. We will also assume normalization 
\eqref{ass1a} and use similar notation as for $p>1$.

We begin with a result similar to Lemma \ref{lem:1}.

\begin{lem}
\label{lem:1s}
Let $X$ be a random variable such that $\Ex |X|^p =1$. Then for every $A>1$ and 
$u,v$ in a normed space $(F,\|\ \|)$ we have
\[
\Ex\|uX+v\|^p\geq \Ex \|uX+v\|^p \ind_{\{ |X|^p \leq A\}} \geq \delta \max \{\|u\|^p, \|v\|^p \},
\]
where
\[
\delta := \Ex(|X|^p-1)\ind_{\{ 1 \leq |X|^p \leq A \}}.
\]
\end{lem}

\begin{proof}
Since $\Ex |X|^p=1$ we have
\begin{equation}
\label{eq:1s1}
\delta
\leq \Ex(|X|^p-1)\ind_{\{ 1 \leq |X|^p  \}} = \Ex(1-|X|^p)\ind_{\{ |X|^p \leq 1 \}}
\leq \Pr(|X|^p \leq 1) \leq \Pr (|X|^p \leq A).
\end{equation}
The triangle inequality yields $\|uX+v\| \geq \big|\|u\||X|-\|v\|\big|$. Thus, it suffices to prove
\begin{equation}
\label{eq:1s2}
\Ex \big|\|u\||X| - \|v\|\big|^p \ind_{\{ |X|^p \leq A\}}  \geq \delta \max \{\|u\|^p, \|v\|^p \}.
\end{equation}
If $u=0$ then this inequality is satisfied due to \eqref{eq:1s1}. In the case $u \neq 0$ 
divide both sides of \eqref{eq:1s2} by $\|u\|^p$ to see that it is enough to show 
\[
\Ex||X| - t|^p  \ind_{\{ |X|^p \leq A\}}  \geq \delta \max \{t^p, 1 \} \quad \mbox{ for } t\geq 0.
\]
To prove this inequality let us consider two cases. First assume that $t \in [0,1]$. Then we have
\begin{align*}
\Ex | |X| - t |^p  \ind_{\{ |X|^p \leq A\}} 
& \geq \Ex | |X| - t |^p  \ind_{\{ 1 \leq |X|^p \leq A\}} \geq \Ex(|X|^p-t^p) \ind_{\{ 1 \leq |X|^p \leq A\}} 
\\
& \geq \Ex(|X|^p-1) \ind_{\{ 1 \leq |X|^p \leq A\}} = \delta=\delta  \max \{t^p, 1 \}. 
\end{align*}

In the case $t>1$ it suffices to note that
\begin{align*}
\Ex | |X| - t |^p  \ind_{\{ |X|^p \leq A\}} 
& \geq \Ex | |X| - t |^p  \ind_{\{ |X|^p \leq 1\}} \geq  \Ex (t^p-|X|^p)  \ind_{\{ |X|^p \leq 1\}} 
\\
& \geq  t^p\Ex ( 1-|X|^p  )  \ind_{\{ |X|^p \leq 1\}} \geq \delta t^p
=\delta  \max \{t^p, 1 \},
\end{align*}
where the last inequality follows from \eqref{eq:1s1}.
\end{proof}

As a consequence, in the same way as in Lemma \ref{lem:3}, we derive the following estimate.

\begin{lem}
\label{lem:3s}
Let r.v.'s $X_1,X_2,\ldots$ satisfy \eqref{ass1a} and \eqref{ass_small2}.  Then for any vectors
$v_0, v_1,\ldots, v_n \in F$ we get 
\[
\Ex \left\|\sum_{i=0}^n v_i R_i\right\|^p \geq \delta^2 \max_{1 \leq i \leq n} \|v_i\|^p 
\geq \frac{\delta^2}{n} \sum_{i=1}^n \|v_i\|^p.
\]
\end{lem}

\begin{lem}
\label{lem:2s}
Suppose that random variables $X_1, X_2,\ldots$ satisfy assumptions \eqref{ass1a} and \eqref{ass_small1}. 
Then for all vectors $v_1, v_2,\ldots$ in $(F,\|\ \|)$ we have
\[
\Pr\left(\left\| \sum_{i=0}^n  v_i R_i \right\|^p  \geq \frac{t}{1-\lambda} \sum_{i=0}^n \lambda^i \|v_i\|^p \right)
\leq \frac{1}{\sqrt{t}} \quad \mbox{for } t>0. 
\]
\end{lem}

\begin{proof}
Note that
\[
\Ex \left\|\sum_{i=0}^n v_i R_i\right\|^{p/2} \leq \sum_{i=0}^n \|v_i\|^{p/2} \Ex |R_i|^{p/2} 
\leq  \sum_{i=0}^n \lambda^i\|v_i\|^{p/2}.
\]	
By the Cauchy-Schwarz inequality we get
\[
\left(  \sum_{i=0}^n  \lambda^i \|v_i\|^{p/2} \right)^2 \leq \sum_{i=0}^n \lambda^i \sum_{i=0}^n \lambda^i \|v_i\|^p
\leq  \frac{1}{1-\lambda}  \sum_{i=0}^n \lambda^i \|v_i\|^p .
\]  
Thus, using Chebyshev's inequality we arrive at
\begin{align*}
\Pr \left(\left\|\sum_{i=0}^n  v_i R_i\right\|^p  \geq \frac{t}{1-\lambda} \sum_{i=0}^n \lambda^i \|v_i\|^p \right)  
&\leq  \Pr\left( \left\|\sum_{i=0}^n  v_i R_i\right\|^{p/2}  \geq \sqrt{t}\sum_{i=0}^n  \lambda^i \|v_i\|^{p/2} \right)
\\
& \leq  \left( \sqrt{t} \sum_{i=0}^n \lambda^i \|v_i\|^{p/2}  \right)^{-1} 
\Ex \left\|\sum_{i=0}^n v_i R_i\right\|^{p/2} \leq \frac{1}{\sqrt{t}}. 
\end{align*}
\end{proof}

Our next lemma is in the spirit of Lemma \ref{lem:4}, but it has a simpler proof.

\begin{lem}
\label{lem:4s}
Let $Y, Z$ be random vectors with values in a normed space $(F,\|\ \|)$ such that 
\[
\Ex \|Z\|^p \ind_{\{ \|Y\|^p \geq \frac{1}{8} \Ex \|Z\|^p  \}} \leq \frac{1}{8} \Ex \|Z\|^p.  
\]
Then
\[
\Ex\|Y+Z\|^p \geq \Ex \|Y\|^p + \frac{1}{2} \Ex\|Z\|^p.
\]
\end{lem}

\begin{proof}
For any $u,v\in F$ we have $\|u+v\|^p\geq \big|\|u\|-\|v\|\big|^p\geq \|u\|^p-\|v\|^p$, therefore
\begin{align*}
\Ex \|Y+Z\|^p 
& \geq \Ex( \|Y\|^p + \|Z\|^p - 2 \|Z\|^p )\ind_{\{ \|Y\|^p  \geq \frac{1}{8} \Ex \|Z\|^p  \}} 
\\ 
& \phantom{aa} + \Ex( \|Y\|^p + \|Z\|^p - 2 \|Y\|^p )\ind_{\{ \|Y\|^p < \frac{1}{8} \Ex \|Z\|^p  \}} 
\\
& \geq  \Ex \|Y\|^p +  \Ex \|Z\|^p -  2\Ex  \|Z\|^p \ind_{\{ \|Y\|^p \geq \frac{1}{8} \Ex \|Z\|^p  \}}-
2\Ex  \|Y\|^p \ind_{\{ \|Y\|^p < \frac{1}{8} \Ex \|Z\|^p  \}}   
\\
& \geq   \Ex \|Y\|^p +  \Ex \|Z\|^p - 2 \cdot \frac{1}{8} \Ex \|Z\|^p -  2 \cdot \frac{1}{8} \Ex \|Z\|^p
=\Ex \|Y\|^p +  \frac{1}{2} \Ex \|Z\|^p.
\end{align*}
\end{proof}

The proof of the lower bound for $p\leq 1$ is similar to the proof for $p>1$ and it relies on a proposition 
similar to Proposition \ref{prop:1}.

\begin{prop}
\label{prop:2}
Let $0<p\leq 1$ and suppose that r.v.'s $X_1,X_2,\ldots$ satisfy assumptions \eqref{ass1a}, \eqref{ass_small1} 
and \eqref{ass_small2}.
Then for any vectors $v_0, v_1,\ldots, v_n$ in a normed space $(F,\|\ \|)$ and any integer $k\geq 1$ we have
\[
\Ex \left\|\sum_{i=0}^n v_i R_i \right\|^p 
\geq  \ve_0 \|v_0\|^p + \sum_{i=1}^n \left(\frac{\ve_1}{k}-c_i\right)\|v_i\|^p,
\]	
where $\ve_0=\delta/8$, $\ve_1 = \delta^3/8$ and
\[
c_i = 0 \mbox{ for }  1 \leq i \leq k-1, \quad c_i = \Phi \sum_{j=k}^i \lambda^j \mbox{ for }  i \geq k
\quad \mbox{and} \quad \Phi = \frac{2^{8} A}{1-\lambda} \lambda^{k-2}.
\]
\end{prop}

\begin{proof}
For $n \leq k$ the assertion follows by Lemmas \ref{lem:1s} and \ref{lem:3s}, since $\ve_0\leq \delta/2$ and  
$\ve_1/k \leq \ve_1/n \leq \delta^2/(2n)$. For $n \geq k$ we proceed by induction on $n$. 

\medskip

\textbf{Case 1.} $\ve_0 \|v_0\|^p \leq \Phi \sum_{i=k}^{n+1}\lambda^i \|v_i\|^p $. 

In this case the induction step is the same as in the proof of Proposition \ref{prop:1}. 

\medskip

\textbf{Case 2.} $\ve_0 \|v_0\|^p > \Phi \sum_{i=k}^{n+1}\lambda^i \|v_i\|^p $. 

Let us define the set
\[
A_k := \{ |X_1|^p \leq A, |R_{2,k}|^p \leq 4 \lambda^{2k-2} \}.
\]
By the induction hypothesis we have
\begin{equation}
\label{row1s}
\Ex \left\|\sum_{i=0}^{n+1} v_i R_i\right\|^p \ind_{\Omega \setminus A_k} \geq  
\ve_0 \Ex \left\| \sum_{i=0}^{k} v_i R_i\right\|^p \ind_{\Omega \setminus A_k} 
+ \sum_{i=k+1}^{n+1} \left(\frac{\ve_1}{k}-c_{i-k} \right) \Ex \|v_i R_k\|^p \ind_{\Omega \setminus A_k}.
\end{equation}

By Chebyshev's inequality and \eqref{ass_small1} we get 
\begin{equation}
\label{eq:estr2k}
\Pr(|R_{2,k}|^p > 4 \lambda^{2k-2})\leq \frac{\Ex |R_{2,k}|^{p/2} }{2\lambda^{k-1}} \leq \frac{1}{2},
\end{equation}
in particular $\Pr(A_k)>0$.
Let $Y,Y',Z$ be defined as in the proof of Proposition \ref{prop:1}. As in \eqref{zv0} we show that Lemma \ref{lem:1s}
yields $\Ex \|Z\|^p \geq \delta \|v_0\|^p$. We have $\|Y\|^p \leq 4A\lambda^{2k-2}\|Y'\|^p$,
variables $Y'$ and $Z$ are independent and $Y'$ has the same distribution as $\sum_{i=k}^{n+1} v_i R_{k+1,i}$.  Thus, 
\begin{align*}
\Ex \|Z\|^p \ind_{\{ \|Y\|^p \geq \frac{1}{8} \Ex \|Z\|^p  \}} 
& \leq  \Ex \|Z\|^p \ind_{\{ 4A \lambda^{2k-2}\|Y'\|^p \geq \frac{\delta}{8} \|v_0\|^p \}} 
\\ 
& = \Ex \|Z\|^p \Pr\left( \|Y'\|^p \geq \frac{1}{4A \lambda^{2k-2}}  \ve_0\|v_0\|^p \right)
\\
& \leq \Ex \|Z\|^p \Pr\left( \left\|\sum_{i=k}^{n+1} v_i R_{k+1,i}\right\|^p 
\geq \frac{2^6}{1-\lambda} \sum_{i=k}^{n+1} \lambda^{i-k} \|v_i\|^p \right) 
\\ 
& \leq \frac{1}{8} \Ex \|Z\|^p,
\end{align*}
where the second inequality follows by the assumptions of Case 2 and the definition of $\Phi$ and the
last one by Lemma \ref{lem:2s}.
Hence, Lemma \ref{lem:4s} yields
\[
\Ex\|Y+Z\|^p \geq \Ex \|Y\|^p + \frac{1}{2} \Ex\|Z\|^p.
\]
Thus,
\begin{equation}
\label{fineq0s}
\Ex \left\| \sum_{i=0}^{n+1} v_i R_i \right\|^p \ind_{A_k} 
\geq \frac{1}{2} \Ex \left\|\sum_{i=0}^{k-1} v_i R_i\right\|^p \ind_{A_k} 
+ \Ex \left\|\sum_{i=k}^{n+1} v_i R_i\right\|^p \ind_{A_k}.
\end{equation}
Using Lemma \ref{lem:1s} and \eqref{eq:estr2k} we obtain
\[
\Ex \left\|\sum_{i=0}^{k-1} v_i R_i\right\|^p \ind_{A_k}  \geq  \delta \|v_0\|^p \Pr(|R_{2,k}|\leq 4\lambda^{2k-2}) 
\geq \frac{\delta}{2} \|v_0\|^p.
\]
Since $\ve_0 \leq \frac{1}{4}$ and $\ve_0 \leq \delta/8$, it follows that
\begin{equation}
\label{fineq1s}
\frac{1}{2} \Ex \left\|\sum_{i=0}^{k-1} v_i R_i\right\|^p \ind_{A_k}  \geq  
\ve_0 \|v_0\|^p + \ve_0 \Ex \left\|\sum_{i=0}^{k-1} v_i R_i\right\|^p \ind_{A_k}.
\end{equation}
By the induction assumption we obtain
\begin{equation}
\label{fineq2s}
\Ex \left\|\sum_{i=k}^{n+1} v_i R_i\right\|^p \ind_{A_k} \geq \ve_0 \Ex \|v_k R_k\|^p \ind_{A_k} 
+ \sum_{i=k+1}^{n+1} \left(\frac{\ve_1}{k} - c_{i-k}\right) \Ex \|v_i R_k\|^p \ind_{A_k}.
\end{equation}
Combining \eqref{fineq0s},  \eqref{fineq1s} and \eqref{fineq2s} we arrive at
\begin{align*}
\Ex\left\| \sum_{i=0}^{n+1} v_i R_i \right\|^p \ind_{A_k}  
&\geq \ve_0 \|v_0\|^p + \ve_0 \Ex \left\|\sum_{i=0}^{k-1} v_i R_i\right\|^p \ind_{A_k} 
+ \ve_0 \Ex \|v_k R_k\|^p \ind_{A_k} 
\\ 
&  \phantom{aa}  + \sum_{i=k+1}^{n+1} \left(\frac{\ve_1}{k} - c_{i-k}\right) \Ex \|v_i R_k\|^p \ind_{A_k} 
\\
&  \geq \ve_0 \|v_0\|^p +  \ve_0 \Ex \left\|\sum_{i=0}^{k} v_i R_i\right\|^p \ind_{A_k}  
+\sum_{i=k+1}^{n+1} \left(\frac{\ve_1}{k} - c_{i-k} \right) \Ex \|v_i R_k\|^p \ind_{A_k}.
\end{align*}
Combining this inequality with \eqref{row1s} yields
\begin{align*}
\Ex\left\| \sum_{i=0}^{n+1} v_i R_i \right\|^p 
& \geq \ve_0 \|v_0\|^p + \ve_0 \Ex \left\|\sum_{i=0}^{k} v_i R_i\right\|^p  
+\sum_{i=k+1}^{n+1} \left(\frac{\ve_1}{k} - c_{i-k} \right) \Ex \|v_i R_k\|^p  
\\
& \geq \ve_0 \|v_0\|^p + \frac{\ve_1}{k}\sum_{i=1}^{k} \|v_i\|^p 
+\sum_{i=k+1}^{n+1} \left(\frac{\ve_1}{k} - c_{i-k} \right)  \|v_i\|^p 
\\
& \geq \ve_0 \|v_0\|^p  +\sum_{i=1}^{n+1} \left(\frac{\ve_1}{k} - c_{i} \right)  \|v_i\|^p,
\end{align*}
where in the second inequality we used Lemma \ref{lem:3s}.
\end{proof}

We are now ready to establish the lower $L_p$-bound for $p\leq 1$.

\begin{proof}[Proof of the lower bound in Theorem \ref{thm_noniid_smallp}]
To show the lower bound let us choose $k$ such that
\[
k\lambda^{2k-2} \leq \frac{\delta^3(1-\lambda)^2}{ 2^{12} A  }.
\] 
Then
\[
c_i \leq \Phi\frac{\lambda^k}{1-\lambda} = \frac{2^{8} A\lambda^{2k-2}}{(1-\lambda)^2} \leq \frac{\ve_1}{2k}.
\]
Therefore, Proposition \ref{prop:2} implies
\[
\Ex \left\|\sum_{i=0}^n v_i R_i\right\|^p \geq \frac{\delta}{8}\|v_0\|^p 
+ \frac{\delta^3}{16k}   \sum_{i=1}^n \|v_i\|^p \geq \frac{\delta^3}{16k} \sum_{i=0}^n \|v_i\|^p.
\]

\end{proof}

\section{Upper bounds}
\label{sec:upper}

The upper bound in Theorem \ref{thm_noniid_smallp} immediately follows by the inequality
$(a+b)^p\leq a^p+b^p$, $a,b\geq 0$, $p\in (0,1]$. To get the upper bound in 
Theorem \ref{thm_noniid_largep} we prove the following result.

\begin{prop}
\label{prop:uppb}
Let $p>0$ and $X_1,X_2,\ldots $ be independent random variables such that
$\Ex|X_i|^p<\infty$ for all $i$ and  
\begin{equation}
\label{eq:assupper}
\forall_{1\leq k<\lceil p\rceil}\ \exists_{\lambda_k<1}\ \forall_{i}\
(\Ex|X_i|^{p-k})^{1/(p-k)} \leq \lambda_k(\Ex|X_i|^{p-k+1})^{1/(p-k+1)}. 
\end{equation}
Then for any vectors $v_0,v_1,\ldots,v_n$ in a normed space $(F,\|\ \|)$ we have
\begin{equation}
\label{eq:upper}
\Ex\left\|\sum_{i=0}^n v_iR_i\right\|^{p}\leq C(p)\sum_{i=0}^n\|v_i\|^p\Ex|R_i|^p,
\end{equation}
where $C(p)=1$ for $p\leq 1$ and for $p> 1$,
\[
C(p)=2^p\left(1+C(p-1)\frac{\lambda_1^{p-1}}{1-\lambda_1^{p-1}}\right)\leq 
2^p\frac{C(p-1)}{1-\lambda_1^{p-1}}.
\]
\end{prop}

\begin{proof}
We have $\|\sum_{i=0}^nv_iR_i\|\leq \sum_{i=0}^n\|v_i\||R_i|$ and $|R_i|=\prod_{j=1}^i|X_j|$, so it is enough to
consider the case when $F=\er$, $v_k\geq 0$ and variables $X_j$ are nonnegative. Since it is only a matter 
of normalization we may also assume that $\Ex X_i^p=1$ for all $i$.

We proceed by induction on $m:=\lceil p\rceil$.
If $m=1$, i.e. $0<p\leq 1$ then the assertion easily follows, since $(x+y)^p\leq x^p+y^p$, $x,y\geq 0$.  

Suppose that $m>1$ and \eqref{eq:upper} holds in the case $p\leq m$. Take $p$ such that
$m< p\leq m+1$. Observe that
\begin{equation}
\label{x+y_to_p}
(x+y)^{p}\leq x^p+2^p(yx^{p-1}+y^p) \quad \mbox{ for } x,y\geq 0.
\end{equation}
Indeed, either $x\leq y$ and then $(x+y)^p\leq 2^p y^p$, or $0\leq y<x$ and then by the convexity of $x^p$,
$((x+y)^p-x^p)/y\leq ((2x)^p-x^p)/x=(2^{p}-1)x^{p-1}$.

We have by \eqref{x+y_to_p}
\[
\Ex\left|\sum_{i=0}^nv_iR_i\right|^{p}\leq \Ex\left|\sum_{i=1}^nv_iR_i\right|^{p}
+2^p\left(v_0\Ex\left|\sum_{i=1}^nv_iR_i\right|^{p-1}+v_0^p\right).
\]
Iterating this inequality we get
\[
\Ex\left|\sum_{i=0}^nv_iR_i\right|^{p}\leq 
v_n^p\Ex R_n^p+2^p\left(\sum_{k=0}^{n-1}v_k\Ex R_{k}\left(\sum_{i=k+1}^nv_iR_{i}\right)^{p-1}
+\sum_{i=0}^{n-1}v_i^p\Ex R_i^p\right).
\]
However, $\Ex R_{k}(\sum_{i=k+1}^nv_iR_{i})^{p-1}=\Ex R_k^p\Ex(\sum_{i=k+1}^nv_iR_{k+1,i})^{p-1}$ and
$\Ex R_k^p=\prod_{j=1}^k \Ex X_j^p=1$. Hence
\[
\Ex\left|\sum_{i=0}^nv_iR_i\right|^{p}\leq 2^p\sum_{i=0}^nv_i^p
+2^p\sum_{k=0}^{n-1}v_k\Ex\left(\sum_{i=k+1}^nv_iR_{k+1,i}\right)^{p-1}.
\]
The induction assumption yields
\begin{align*}
\Ex\left(\sum_{i=k+1}^nv_iR_{k+1,i}\right)^{p-1} &\leq C(p-1)\sum_{i=k+1}^nv_i^{p-1}\Ex R_{k+1,i}^{p-1}=
C(p-1)\sum_{i=k+1}^nv_i^{p-1}\prod_{j=k+1}^i\Ex X_j^{p-1}
\\
&\leq C(p-1)\sum_{i=k+1}^nv_i^{p-1}\lambda_1^{(p-1)(i-k)},
\end{align*}
where the last inequality follows by \eqref{eq:assupper}.
To finish the proof we observe that
\begin{align*}
\sum_{k=0}^{n-1}v_k\sum_{i=k+1}^nv_i^{p-1}\lambda_1^{(p-1)(i-k)}&\leq
\sum_{0\leq k< i\leq n}\left(\frac{1}{p}v_k^p+\frac{p-1}{p}v_i^p\right)\lambda_1^{(p-1)(i-k)}
\\
&\leq \sum_{i=0}^nv_i^p\sum_{j=1}^{\infty}\lambda_1^{(p-1)j}
= \frac{\lambda_1^{p-1}}{1-\lambda_1^{p-1}}\sum_{i=0}^n v_i^p. 
\end{align*} 

\end{proof}

\noindent
{\bf Remark.} It is not hard to show by induction on $\lceil p\rceil$ that
\[
C(p)\leq 2^{\frac{p(p+1)}{2}}\prod_{1\leq j\leq \lceil p\rceil -1}\frac{1}{1-\lambda_j^{p-j}}.
\]

\section{Stochastic recursions}
\label{sec:perp}

The proof of Theorem \ref{mthmper} is only a slight modification of the proof of Theorem \ref{thm_main}. 
Normalizing we may always assume $\Ex X^p=1$. The upper bound follows as in the proof of Proposition \ref{prop:uppb} 
(see more details below). To show the lower bound we consider two cases:
\begin{equation}
\tag{C1}
\mbox{There are $w,u \in F$ such that $w+B+Xu =0$ a.e.}
\end{equation}
or
\begin{equation}
\tag{C2}
\Pr(w+B+Xu=0)<1 \mbox{ for every } w,u \in F.
\end{equation}

In case (C1) we get
\begin{align*}
\sum_{i=1}^nR_{i-1}B_i
&=\sum _{i=1}^nR_{i-1}(-w-X_iu)
=-\sum _{i=1}^nR_{i-1}w-\sum _{i=1}^nR_iu
\\
&=-\sum _{i=1}^{n}R_{i-1}(w+u)+u-R_nu.
\end{align*}
Notice that 
\[
\Ex \left\| \sum _{i=1}^{n}R_{i-1}(w+u)\right\| ^p=
\| w+u\| ^p\Ex \left| \sum _{i=1}^{n}R_{i-1}\right| ^p\geq c_{p,X}n\| w+u\| ^p,
\]
where the last inequality follows by Theorem \ref{thm_main} with $F=\er$ and $v_i=1$. 
Assumption \eqref{eq:nondeg} implies $w+u\neq 0$. Moreover, 
\[
\E \| u-R_nu\| ^p \leq 2^p\| u\| ^p(1+\Ex R_n^p)=2^{p+1}\|u\|^p.
\]
Hence for $n\geq n_0=n_0(X,B)$ and $c=c(p,X,B)=\frac{1}{2^{p+1}}c_{p,X}\| w+u\| ^p$,
\[
\Ex \left\|\sum_{i=1}^nR_{i-1}B_i\right\|^p=
\Ex \left\| \sum _{i=1}^{n}R_{i-1}(w+u)-(u-R_nu)\right\| ^p\geq cn.
\]
To get the lower bound in \eqref{bounds} for $1\leq n<n_0$ we observe that 
\begin{align*}
cnn_0&\leq \Ex \left\|\sum_{i=1}^{nn_0}R_{i-1}B_i\right\|^p=
\Ex \left\|\sum_{k=0}^{n_0-1}\sum_{i=kn+1}^{(k+1)n}R_{i-1}B_i\right\|^p
\\
&\leq n_0^{p}\sum_{k=0}^{n_0-1}\Ex \left\|\sum_{i=kn+1}^{(k+1)n}R_{i-1}B_i\right\|^p
=n_0^{p}\sum_{k=0}^{n_0-1}\Ex R_{kn}^p \Ex\left\|\sum_{i=kn+1}^{(k+1)n}R_{kn+1,i-1}B_i\right\|^p
\\
&=n_0^{p}n_0\Ex\left\|\sum_{i=1}^{n}R_{i-1}B_i\right\|^p,
\end{align*}
where the last equality follows since $\sum_{i=kn+1}^{(k+1)n}R_{kn+1,i-1}B_i$ has the same distribution as
$\sum_{i=1}^{n}R_{i-1}B_i$.

It is worth mentioning that the estimate $\Ex( \sum _{i=1}^{n}R_{i-1}) ^p\geq cn$ was first observed in 
\cite{BJMW} under the \emph{Goldie-Kesten conditions}. In fact, a stronger statement  was proved there: 
$\lim _{n\to \infty} \frac{1}{n}\E ( \sum _{i=1}^{n}R_{i-1}) ^p $ exists and it is strictly positive. 
Note also that if $u=-w$, i.e. assumption \eqref{eq:nondeg} is not satisfied, then
\[
\Ex\left\|\sum_{i=1}^nR_{i-1}B_i\right\|^p=\Ex\|u-R_nu\|^p\leq 2^{p+1}\|u\|^p
\]
and the lower bound in \eqref{bounds} cannot hold for large $n$. 

In the sequel, to derive the lower bound it is enough to consider case (C2). The following lemma is then a counterpart of
Lemmas \ref{lem:1} and \ref{lem:1s}.

\begin{lem}
\label{mlem}
Suppose that $X$ is a nonnegative, nondegenerate r.v., $B$ is a random vector with values in a separable Banach space $F$,
$\Ex X^p,\Ex\|B\|^p<\infty$ and for any $u,w\in F$, $\Pr(B+Xu=w)<1$. Then there exist
constants $A<\infty$ and $\delta>0$, depending only on the distribution of $(B,X)$ and $p$, such that
\[
\Ex\|w+B+Xu\|^p\ind_{\{X\leq A\}}\geq \delta \max\{\|w\|^p,\|u\|^p,\Ex\|B\|^p\}.
\] 
\end{lem}

\begin{proof}
By $\delta_1$ and $\delta_2$  we will denote in the sequel positive constants depending only on the distribution of 
$(B,X)$ and $p$. Lemmas \ref{lem:1} and \ref{lem:1s} yield
\[
\Ex\|w+Xu\|^p\geq \delta_1 \max\{\|w\|^p,\|u\|^p\} \quad \mbox{ for any } w,u\in F.
\] 
Since $\|u_1+u_2\|^p\leq 2^p(\|u_1\|^p+\|u_2\|^p)$ for any $u_1,u_2\in F$, we get
\[
\Ex\|w+B+Xu\|^p\geq 2^{-p}\Ex\|w+Xu\|^p-\Ex\|B\|^p\geq  2^{-p-1}\delta_1 \max\{\|w\|^p,\|u\|^p,\Ex\|B\|^p\},
\]
provided that $\max\{\|w\|^p,\|u\|^p\}\geq M:=2^{p+1}\max\{1,\delta_1^{-1}\}\Ex\|B\|^p$.
Let
\[
\alpha:=\inf\left\{\Ex\|w+B+Xu\|^p\colon\ \max\{\|w\|^p,\|u\|^p\}\leq M\right\}.
\]

First we observe that $\alpha>0$. Indeed, assume that $\alpha=0$. Then there exist sequences 
$(u_n)$, $(w_n)$ in $F$ such
that $\|u_n\|^p\leq M$, $\|w_n\|^p\leq M$ and $\Ex\|w_n+B+Xu_n\|^p\rightarrow 0$. We have 
\begin{align*}
\Ex\|w_n+B+Xu_n\|^p+\Ex\|w_m+B+&Xu_m\|^p
\\
&\geq 2^{-p}\Ex\|(w_n+B+Xu_n)-(w_m+B+Xu_m)\|^p
\\
&\geq 2^{-p}\delta_1 \max\{\|w_n-w_m\|^p,\|u_n-u_m\|^p\}.
\end{align*}
Thus both sequences $(u_n)$ and  $(w_n)$ satisfy the Cauchy condition, hence they are convergent, respectively
to $u$ and $w$. But then $\Ex\|w+B+Xu\|^p=\lim_n\Ex\|w_n+B+Xu_n\|^p=0$, which contradicts our assumptions.

Therefore $\alpha>0$ and for  $\max\{\|w\|^p,\|u\|^p\}\leq M$ we get
\[
\Ex\|w+B+Xu\|^p \geq \alpha \geq 
\alpha\max\left\{\frac{1}{M}\|w\|^p,\frac{1}{M}\|u\|^p,\frac{1}{\Ex\|B\|^p}\Ex\|B\|^p\right\}.
\]
This way we showed that
\[
\Ex\|w+B+Xu\|^p\geq \delta_2\max\{\|w\|^p,\|u\|^p,\Ex\|B\|^p\} \quad \mbox{ for any } w,u\in F.
\]
To finish the proof it is enough to note that
\[
\Ex\|w+B+Xu\|^p\ind_{\{X> A\}}\leq 3^p\Ex(\|w\|^p+\|B\|^p+\|u\|^p)\ind_{\{X> A\}}\leq 
\frac{\delta_2}{2}\max\{\|w\|^p,\|u\|^p,\Ex\|B\|^p\},
\]
provided that $A$ is large enough.
\end{proof}
 
For the rest of the proof of the lower bound in \eqref{bounds} we do not need to assume that $(X_i,B_i)$ are i.i.d,
but we need uniformity in Lemma \ref{mlem}, i.e. the condition 
\begin{equation}
\label{basic}
\exists_{\delta>0, A<\infty}\
\forall_{i}\ \forall_{w,u\in F}\
\Ex \|w+B_i+X_iu \| ^p\ind_{\{X_i^p\leq A\Ex X_i^p\}} \geq \delta 
\max \{\Ex \| B_i\| ^p, \| w\| ^p, \| u\| ^p\Ex X_i^p\}.
\end{equation}  
More precisely, the following theorems hold.

\begin{thm}
\label{mthmB}
Let $0<p\leq 1$ and let $(X_1,B_1), (X_2,B_2)...\in \R ^+\times F$ be a  sequence of independent random variables 
such that $\Ex \| B_i\| ^p, \Ex X_i^p <\infty$. Suppose that conditions \eqref{ass_small1} and \eqref{basic} 
are satisfied.
Then there is a constant  $c(p,\lambda , \delta , A)$ such that for every $n$,
\begin{equation}
\label{result1}
c(p,\lambda , \delta , A)\sum _{i=1}^n(\Ex R_{i-1}^p)\Ex \| B_i\| ^p
\leq \Ex \left\| \sum _{i=1}^n R_{i-1} B_i\right\| ^p
\leq   \sum _{i=1}^n(\Ex R_{i-1}^p)\Ex \| B_i\| ^p.
\end{equation}
\end{thm}

\begin{thm}
\label{mthmBB}
Let $p>1$ and let $(X_1,B_1), (X_2,B_2)...\in \R ^+\times F$ be a  sequence of independent random variables 
such that $\Ex \| B_i\| ^p, \Ex X_i^p <\infty$. Suppose that conditions \eqref{ass_large2}, \eqref{ass_large3} and
\eqref{basic}  are satisfied.   
Then there are constants  $c = c(p,q,\lambda , \delta , A),C(p,\lambda_1,\dots \lambda_{\lceil p\rceil -1})$ such that 
for every $n$,
\begin{equation}
\label{result2}
c(p,\lambda, \delta , A)\sum _{i=1}^n(\Ex R_{i-1}^p)\Ex \| B_i\| ^p
\leq \Ex \left\| \sum _{i=1}^n R_{i-1} B_i\right\| ^p
\leq C(p,\lambda_1,\dots \lambda_{\lceil p\rceil -1}) \sum _{i=1}^n(\Ex R_{i-1}^p)\Ex \| B_i\| ^p.
\end{equation}
\end{thm}

Since it is only a matter of normalization we may and will assume that $\Ex X_i^p=1$.

First we prove the upper bound in \eqref{result2}. 
Proceeding by induction, as in the proof of Proposition \ref{prop:uppb}, we get
\begin{align*}
\Ex \left\| \sum _{i=1}^nR_{i-1}B_i\right\| ^p
&\leq \Ex \left\| \sum _{i=2}^nR_{i-1}B_i\right\| ^p 
+2^p\left(\Ex \| B_1\| \left(\sum _{i=2}^nR_{i-1}\| B_i\| \right)^{p-1} +\Ex \| B_1\| ^p\right )
\\
&=\E \left\| \sum _{i=2}^nR_{i-1}B_i\right\| ^p 
+2^p\left (\Ex \| B_1\|X_1^{p-1} \left(\sum _{i=2}^nR_{2,i-1}\| B_i\| \right)^{p-1} +\Ex \| B_1\| ^p\right).
\end{align*}  
Iterating this inequality we obtain
\begin{align*}
\Ex &\left\| \sum _{i=1}^nR_{i-1}B_i\right\| ^p
\\
&\leq \Ex \| B_n\| ^p+ 2^p\sum _{k=1}^{n-1}\Ex \| B_k\|R_{k-1} R_k^{p-1}\left( \sum _{i=k+1}^nR_{k+1,i-1}\| B_i\|\right)^{p-1} 
+2^p\sum _{i=1}^{n-1}\Ex\| B_i\| ^p
\\
&\leq 2^p\sum _{i=1}^{n}\Ex\| B_i\| ^p
+2^p\sum _{k=1}^{n-1}\Ex \| B_k\| X_k^{p-1}\Ex\left( \sum _{i=k+1}^nR_{k+1,i-1}\| B_i\|\right) ^{p-1}.
\end{align*}
By the induction assumption 
\begin{align*}
\Ex \left( \sum _{i=k+1}^nR_{k+1,i-1}\| B_i\|\right ) ^{p-1}
&\leq C(p-1) \sum _{i=k+1}^n(\E R_{k+1,i-1}^{p-1})\Ex \| B_i\| ^{p-1}
\\
&\leq C(p-1) \sum _{i=k+1}^n\lambda _1^{(i-1-k)(p-1)}\Ex \| B_i\| ^{p-1}.
\end{align*}
Hence,
\[
\Ex \left\| \sum _{i=1}^nR_{i-1}B_i\right\| ^p
\leq 2^p\sum _{i=1}^{n}\Ex\| B_i\| ^p+
2^pC(p-1)\sum _{k=1}^{n-1}\sum _{i=k+1}^n\lambda_1^{(i-1-k)(p-1)}\Ex \| B_k\| X_k^{p-1}\| B_i\| ^{p-1}.
\]
To finish the proof of the upper bound we observe that for $k<i$,
\[
\Ex \| B_k\| X_k^{p-1}\| B_i\| ^{p-1}
\leq \frac{1}{p}\Ex(\| B_k\|^p+ (p-1)X_k^{p}\| B_i\| ^{p})=\frac{1}{p}(\Ex \| B_k\|^p+ (p-1)\Ex \| B_i\| ^{p}).
\]
Therefore,
\begin{align*}
\sum _{k=1}^{n-1}\sum _{i=k+1}^n&\lambda_1^{(i-1-k)(p-1)}\Ex \| B_k\| X_k^{p-1}\| B_i\| ^{p-1}
\\
&\leq \sum _{k=1}^{n-1}\sum _{i=k+1}^n\lambda_1^{(i-1-k)(p-1)}
\left(\frac{1}{p}\E \| B_k\|^p+ \frac{p-1}{p}\E \| B_i\| ^{p}\right)
\leq \frac{1}{1-\lambda_1^{p-1}}\sum _{i=1}^{n}\Ex\| B_i\| ^p
\end{align*}
and the conclusion follows.

To prove the lower bounds in \eqref{result1} and \eqref{result2} we follow closely arguments of 
Sections \ref{sec:lowerlarge} and \ref{sec:lowersmall}, making use of \eqref{basic}  whenever 
Lemma \ref{lem:1} or Lemma \ref{lem:1s} are used. For instance, to obtain the estimate 
\begin{equation}
\label{eq:1sp1}
\Ex \left\| w+ \sum _{i=1}^{n}R_{i-1}B_i\right\| ^p \geq \delta \max_{1\leq j\leq n}\Ex \| B_j\| ^p
\geq \frac{\delta}{n}\sum_{j=1}^n\Ex \| B_j\| ^p
\end{equation}
we proceed as follows. For $1\leq j\leq n$ we have
\[
\Ex \left\|  w+ \sum _{i=1}^{n}R_{i-1}B_i\right\| ^p
\geq \Ex \left\|  w+ \sum _{i=1}^{n}R_{i-1}B_i\right\| ^p\ind_{\{R_{j-1}>0\}}
=\Ex R_{j-1}^p \| Y_j+B_j+X_jZ_j\| ^p,
\]
where  
\[
Y_j:=\left(w+\sum _{i=1}^{j-1}R_{i-1}B_i\right)\frac{1}{R_{j-1}}\ind_{\{R_{j-1}>0\}}
\quad\mbox{ and }\quad
Z_j:=\sum _{i=j+1}^nR_{j+1,i-1}B_i.
\]
Since variables $R_{j-1},Y_j$ and $Z_j$ are independent of $(X_j,B_j)$, condition \eqref{basic} yields
\[
\Ex \left\|  w+ \sum _{i=1}^{n}R_{i-1}B_i\right\| ^p\geq \delta \Ex R_{j-1}^p\Ex \| B_j\|^p=\delta\Ex \| B_j\| ^p.
\]
Similar argument used for $j=1$ yields 
\begin{equation}
\label{eq:1sp2}
\Ex \left\|  w+ \sum _{i=1}^{n}R_{i-1}B_i\right\| ^p\geq \delta \| w\| ^p.
\end{equation}

For the rest of this section let us concentrate on the case $p\leq 1$,
presenting only the parts of the argument that are specific 
for the setting of Theorem \ref{mthmB}. If $p>1$ the argument is completely analogous. In this situation 
Lemma \ref{lem:2s}  holds with the same proof.

\begin{lem} 
Suppose the assumptions of Theorem \ref{mthmB} are satisfied. Then for $t>0$,
\[
\Pr \left(\left\| \sum _{i=1}^{n}X_1\dots X_{i-1}B_i\right\| ^p\geq 
\frac{t}{1-\lambda}\sum _{i=1}^{n}\lambda^{i-1}\Ex \|B_i\| ^p \right)\leq t^{-1\slash 2}.
\]
\end{lem}

The main proposition (analogous to Proposition \ref{prop:2}) can be formulated as follows.
\begin{prop}
\label{mainprop}
Suppose that the assumptions of Theorem \ref{mthmB} are satisfied and $\Ex X_i^p=1$ for all $i$. Then for any $w\in F$ and
$k=1,2,\ldots$ we have
\[
\Ex \left\| w+ \sum _{i=1}^{n}R_{i-1}B_i\right\| ^p
\geq \ve_0\| w\| ^p+ \sum _{i=1}^{n}\left(\frac{\ve_1}{k}-c_i\right)\Ex \| B_i\| ^p,
\]
where $\ve_0=\delta/8$, $\ve_1=\delta\ve_0$,
\[
c_i=0\mbox{ for } 1\leq i\leq k-1, \quad
c_i=\Phi \sum _{j=k}^i\lambda ^{j-1}, i\geq k\quad \mbox{ and }\quad  \Phi =\frac{2^{8}A}{(1-\lambda )}\lambda ^{k-2}. 
\]
\end{prop}

\begin{proof}
For $n\leq k$ the assertion follows by \eqref{eq:1sp1} and \eqref{eq:1sp2}. For $n\geq k$ we proceed by induction. 
To simplify the notation let for $k=1,2,\ldots$ and $w\in F$,
\[
S_{k,n}(w):=w+\sum_{i=k}^nR_{k,i-1}B_i\quad \mbox{ and }\quad S_n(w):=S_{1,n}(w)=w+\sum_{i=1}^nR_{i-1}B_i.
\]
Observe that the random variable $S_{k,n}(w)$ is independent of $(X_i,B_i)_{i\leq k-1}$.

As in the proof of Proposition \ref{prop:2} we consider two cases. First assume that
\begin{equation}
\label{case1}
\ve_0\| w\| ^p\leq \Phi \sum _{i=k}^{n+1}\lambda^{i-1}\Ex \| B_i\| ^p.
\end{equation}
We have
\[
\Ex \| S_{n+1}(w)\| ^p =\Ex  X_1^p\| S_{2,n+1}(w')\| ^p\ind_{\{X_1>0\}}+
\Ex \| w+B_1\| ^p\ind_{\{X_1=0\}},
\]
where $w'=X_1^{-1}(w+B_1)\ind_{\{X_1>0\}}$. Hence by the induction assumption (used conditionally on $(X_1,B_1)$)
we get
\begin{align*}
\Ex \| S_{n+1}(w)\| ^p
&\geq \Ex X_1^p\left (\ve _0\| w'\| ^p + \sum _{i=2}^{n+1}\left(\frac{\ve_1}{k}-c_{i-1}\right)\Ex \| B_i\| ^p\right)\ind_{\{X_1>0\}}
+\Ex\| w+B_1\| ^p\ind_{\{X_1=0\}}
\\
&=\ve_0\Ex\| w+B_1\| ^p\ind_{\{X_1>0\}}+
\sum _{i=2}^{n+1}\left(\frac{\ve_1}{k}-c_{i-1}\right)\Ex X_1^p\| B_i\| ^p\ind_{\{X_1>0\}}
\\
&\phantom{=}+\Ex \| w+B_1\| ^p\ind_{\{X_1=0\}}
\\
&\geq \ve_0\Ex \| w +B_1\| ^p + \sum _{i=2}^{n+1}\left(\frac{\ve _1}{k}-c_{i-1}\right)\Ex \| B_i\| ^p 
\\
&\geq  \ve_1\Ex \| B_1\| ^p + \sum _{i=2}^{n+1}\left(\frac{\ve_1}{k}-c_{i-1}\right)\Ex \| B_i\| ^p 
+ \ve _0\| w\| ^p - \Phi \sum _{i=k}^{n+1}\lambda ^{i-1}\E \| B_i\| ^p
\\
&= \ve_0\| w\| ^p+\ve_1\Ex \| B_1\| ^p + \sum _{i=2}^{n+1}\left(\frac{\ve_1}{k}-c_{i}\right)\Ex \| B_i\| ^p,
\end{align*}
where we used independence of $X_1$ and $B_i$ for $i\geq 2$, normalization $\Ex X_1^p\ind_{\{X_1>0\}}=\Ex X_1^p=1$
and inequalities \eqref{eq:1sp1} and \eqref{case1}.

Now suppose that
\[
\ve_0\| w\| ^p> \Phi \sum _{i=k}^{n+1}\lambda^{i-1}\E \| B_i\| ^p
\]
and let
\[
U_k:=\{ X_1^p\leq A, R^p_{2,k}\leq 4\lambda^{2k-2}\}.
\]
We have
\begin{align*}
\Ex \| S_{n+1}(w)\| ^p\ind_{\Om \setminus U_k}
& =\Ex  R_{k}^p\| S_{k+1,n+1}(w')\| ^p\ind_{\Om \setminus U_k}\ind_{\{R_k>0\}}
\\
&\phantom{=}+\Ex \| w+\sum _{i=1}^{k}R_{i-1}B_i\| ^p\ind_{\Om \setminus U_k}\ind_{\{R_k=0\}},
\end{align*}
where $w'=(R_{k})^{-1}(w+ \sum _{i=1}^{k}R_{i-1}B_i)\ind_{\{R_k>0\}}$.
Hence by the induction assumption
\begin{align*}
\Ex \| S_{n+1}&(w)\| ^p\ind_{\Om \setminus U_k}
\\
\geq& \Ex R_k^{p}\left(\ve_0 \| w'\| ^p + \sum _{i=k+1}^{n+1}\left(\frac{\ve_1}{k}-c_{i-k}\right)\E \| B_i\| ^p\right)
\ind_{\Om \setminus U_k}\ind_{\{R_k>0\}}
\\
&+\Ex \left\| w+\sum _{i=1}^{k}R_{i-1}B_i\right\| ^p\ind_{\Om \setminus U_k}\ind_{\{R_k=0\}}
\\
=&  \ve_0 \Ex \left\| w+ \sum _{i=1}^{k}R_{i-1}B_i \right\| ^p\ind_{\Om \setminus U_k}\ind_{\{R_k>0\}}
\\
&+\Ex R_k^p\sum _{i=k+1}^{n+1}\left(\frac{\ve_1}{k}-c_{i-k}\right)\Ex \| B_i\| ^p\ind_{\Om \setminus U_k}\ind_{\{R_k>0\}}
\\
&+\Ex \left\| w+\sum _{i=1}^{k}R_{i-1}B_i\right\| ^p\ind_{\Om \setminus U_k}\ind_{\{R_k=0\}}
\\
\geq & \ve_0 \Ex \left\| w+ \sum _{i=1}^{k}R_{i-1}B_i \right\| ^p\ind_{\Om \setminus U_k}
+\Ex R_k^p\sum_{i=k+1}^{n+1}\left(\frac{\ve_1}{k}-c_{i-k}\right)\Ex \| B_i\| ^p\ind_{\Om \setminus U_k}
\\
=& \ve _0 \Ex \left\| w+ \sum _{i=1}^{k}R_{i-1}B_i \right\| ^p\ind_{\Om \setminus U_k}
+\sum _{i=k+1}^{n+1}\left(\frac{\ve _1}{k}-c_{i-k}\right)\Ex \|R_k B_i\| ^p\ind_{\Om \setminus U_k}.
\end{align*}

To finish the proof we define $(Z,Y,Y')$ as the random variable 
\[
\left(w+ \sum _{i=1}^{k}X_{k}\dots X_{i-1}B_i, 
\sum _{i=k+1}^{n+1}X_{1}\dots X_{i-1}B_i, \sum _{i=k+1}^{n+1}X_{k+1}\dots X_{i-1}B_i\right)
\] conditioned on $U_k$ and we
proceed as in the proof of Proposition \ref{prop:2}.
\end{proof}

\noindent
{\sc Ewa Damek}\\
Institute of Mathematics\\
University of Wroc{\l}aw\\
Pl. Grunwaldzki 2/4\\
50-384 Wroc{\l}aw, Poland\\
\texttt{edamek@math.uni.wroc.pl}

\medskip
\noindent
{\sc Rafa{\l} Lata{\l}a, Piotr Nayar}\\
Institute of Mathematics\\
University of Warsaw\\
Banacha 2\\
02-097 Warszawa, Poland\\
\texttt{rlatala@mimuw.edu.pl, nayar@mimuw.edu.pl}

\medskip
\noindent
{\sc Tomasz Tkocz}\\
Mathematics Institute\\
University of Warwick\\
Coventry CV4 7AL, UK\\
\texttt{t.tkocz@warwick.ac.uk}


\begin{thebibliography}{11}
\bibitem{AB}
D.~J.~Aldous and A.~Bandyopadhyay,
A survey of max-type recursive distributional equations.
\emph{Ann. Appl. Probab.} \textbf{15} (2005), 1047--1110.

 
\bibitem{AM}
G.~Alsmeyer and S.~Mentemeier,
Tail behavior of stationary solutions of random difference equations: the case of regular matrices. 
\emph{J. Difference Equ. Appl.}  \textbf{18} (2012), 1305--1332.

\bibitem{BBE}
M.~Babillot, P.~Bougerol, and L.~Elie,
The random difference equation {$X_n=A_nX_{n-1}+B_n$} in the critical case,
\emph{Ann. Probab.} \textbf{25} (1997), 478--493.

\bibitem{BJMW} 
K.~Bartkiewicz, A.~Jakubowski, T.~Mikosch and O~Wintenberger,
Stable limits for sums of dependent infinite variance random variables,
\emph{Prob. Theory Related Fields} \textbf{150} (2011), 337--372.

\bibitem{BDGHU}
D.~Buraczewski, E.~Damek, Y.~Guivarch, A.~Hulanicki and R.~Urban,
Tail-homo\-gene\-ity of stationary measures for some multidimensional stochastic recursions,
\emph{Probab. Theory Related Fields} \textbf{145} (2009), 385--420.

\bibitem{BDZ}
D.~Buraczewski, E.~Damek and J.~Zienkiewicz,
On the Kesten-Goldie constant,
submitted.

\bibitem{CV}
J.~F.~Collamore and A.~N.~Vidyashankar,
Tail estimates for stochastic fixed point equations via nonlinear renewal theory,
\emph{Stochastic Process. Appl.} \textbf{123} (2013), 3378--3429.

\bibitem{DF} P.~Diaconis and D.~Freedman,
Iterated random functions,
\emph{SIAM Rev.} {\bf 41} (1999), 45--76.

\bibitem{EKM}
P.~Embrechts, C.~Kl\"uppelberg and T.~Mikosch,
\emph{Modelling extremal events. For insurance and finance},
Springer, Berlin  1997.


\bibitem{ESZ}
N.~Enriquez, C.~Sabot and O.~Zindy,
A probabilistic representation of constants in Kesten's renewal theorem,
\emph{Probab. Theory Related Fields}  \textbf{144} (2009), 581--613.

\bibitem{G}
C.~M.~Goldie,
Implicit renewal theory and tails of solutions of random equations,
\emph{Ann. Appl. Probab.}, \textbf{1} (1991), 126--166.

\bibitem{GG}
C.~M.~Goldie and R.~Gr\"ubel,
Perpetuities with thin tails,
\emph{Adv. in Appl. Probab.}  \textbf{28} (1996),  463--480.


\bibitem{Gui} 
Y.~Guivarc'h,
Heavy tail properties of multidimensional stochastic recursions,
\emph{Dynamics \& stochastics}, 85--99,
IMS Lecture Notes Monogr. Ser. 48, IMS, Beachwood 2006.

\bibitem{HW}
P.~Hitczenko and J.~Weso{\l}owski,
Perpetuities with thin tails revisited,
\emph{Ann. Appl. Probab.}  \textbf{19} (2009), 2080--2101.

\bibitem{K}
H.~Kesten,
Random difference equations and renewal theory for products of random matrices,
\emph{Acta Math.}, \textbf{131} (1973), 207--248.

\bibitem{KP} 
C.~Kl\"uppelberg and S.~Pergamenchtchikov,
The tail of the stationary distribution of a random coefficient ${\rm AR}(q)$ model,
\emph{Ann. Appl. Probab.} \textbf{14} (2004), 971--1005.

\bibitem{La} 
R.~Lata{\l}a,
$L_1$-norm of combinations of products of independent random variables,
\emph{Israel J. Math.} \textbf{203} (2014), 295--308.


\bibitem{YM} 
Y.~Meyer, 
Endomorphismes des id\'eaux ferm\'es de $L^{1}\,(G)$, classes de Hardy et  s\'eries de Fourier lacunaires,
\emph{Ann. Sci. \'Ecole Norm. Sup. (4)}  \textbf{1} (1968), 499--580. 

\bibitem{RS}
S.~T. Rachev and G. Samorodnitsky,
Limit laws for a stochastic process and random recursion arising in probabilistic modelling,
\emph{Adv. in Appl. Probab.}, \textbf{27} (1995), 185--202.

\bibitem{V} 
W.~Vervaat,
On a stochastic difference equation and a representation of non-ne\-ga\-ti\-ve infinitely divisible random variables,
{\em Adv. in Appl. Prob.} \textbf{11} (1979), 750--783.


\end{thebibliography}
\end{document}